\documentclass[reqno]{amsart}
\usepackage{amsthm,amssymb,amsmath}
\usepackage{mathtools}
\usepackage{tikz}
\usetikzlibrary[matrix,arrows,arrows.meta,calc,decorations.pathmorphing,backgrounds,positioning,fit,petri]
\usepackage{amscd}   % for commutative diagrams
\usepackage[all,graph,poly]{xy} % for complicated commutative diagrams
\newcommand{\abs}[1]{\lvert#1\rvert}
\newcommand{\bigabs}[1]{\big\lvert#1\big\rvert}
\newcommand{\B}{\mathbb{B}}
\newcommand{\bb}[1]{\mathbb{#1}}
\newcommand{\bwg}{\msf{BW}^{\mathcal{G}}_{16}}
\newcommand{\cC}{\mathcal{C}}
\newcommand{\cE}{\mathcal{E}}
\newcommand{\cG}{\mathcal{G}}
\newcommand{\cH}{\mathcal{H}}
\newcommand{\cK}{\mathcal{K}}
\newcommand{\cO}{\mathcal{O}}

\newcommand{\even}{\mathrm{I I}}
\newcommand{\fF}{\mathfrak{F}}
\newcommand{\fK}{\mathfrak{K}}
\newcommand{\cell}{\mathcal{G}_{1,1}}
\newcommand{\C}{\mathbb{C}}
\newcommand{\F}{\mathbb{F}}
\newcommand{\K}{\mathbb{K}}

\newcommand{\ph}{\phantom{-}}

\newcommand{\Q}{\mathbb{Q}}
\newcommand{\R}{\mathbb{R}}
\newcommand{\T}{\mathbb{T}}
\newcommand{\Z}{\mathbb{Z}}
\newcommand{\ip}[2]{ \langle  #1 \vert #2  \rangle }
\newcommand{\bigip}[2]{ \big\langle  #1 \big\vert #2  \big\rangle }

\newcommand{\msf}[1]{\mathsf{#1}}

\newcommand{\op}[1]{\operatorname{\msf{#1}}}
\swapnumbers
\newtheorem{theorem}{Theorem}[section]
\newtheorem{lemma}[theorem]{Lemma}
\newtheorem{corollary}[theorem]{Corollary}
\theoremstyle{definition}
\newtheorem{definition}[theorem]{Definition}

\newtheorem{topic}[theorem]{}
\theoremstyle{remark}
\newtheorem{remark}[theorem]{Remark}
\begin{document}
%
%*******************************************************************************************************
%
%
%*******************************************************************************************************
%
\title[]{A new complex reflection group in $PU(9,1)$ and the Barnes-Wall lattice}
\author{Tathagata Basak}
\address{Department of Mathematics\\Iowa State University, \\Ames, IA 50011}
\email{tathagat@iastate.edu}
\urladdr{http://orion.math.iastate.edu/tathagat}
\date{April 12, 2018}
\begin{abstract}
We show that the projectivized complex reflection group
$\Gamma$ of the unique $(1+i)$-modular Hermitian 
$\Z[i]$-module of signature $(9,1)$ is a new arithmetic reflection group in $PU(9,1)$. 
We find $32$ complex reflections of order four generating $\Gamma$.
The mirrors of these $32$ reflections form the vertices of a 
sort of Coxeter-Dynkin diagram
$D$ for $\Gamma$ that encode Coxeter-type generators
and relations for $\Gamma$. The vertices of $D$ can be indexed by
sixteen points and sixteen affine hyperplanes
in $\F_2^4$. The edges of $D$ are determined by the finite geometry of these
points and hyperplanes. 
The group of automorphisms of the diagram $D$ is
$2^4 \colon (2^3 \colon L_3(2)) \colon 2$. This group transitively permutes
the $32$ mirrors of generating reflections 
and fixes an unique point $\tau$ in $\C H^9$. These $32$ mirrors
are precisely the mirrors closest to $\tau$.
These results are strikingly similar to the results satisfied by the complex
hyperbolic reflection group at the center of Allcock's monstrous proposal.
\end{abstract}
\maketitle
%
%*******************************************************************************************************
%
\section{Introduction}
Let $\cG = \Z[i]$ be the ring of Gaussian integers. Let $p = (1 + i)$. 
We study the projectivized complex reflection group $\Gamma = \Gamma_1$
of the unique $p$-modular
Hermitian $\cG$--lattice of signature $(9,1)$. In particular, we show that $\Gamma_1$
is arithmetic and we find nice generators and relations for $\Gamma_1$
(The notation $\Gamma_1$ is only used in the introduction. Afterwards
we shall simply write $\Gamma$ instead of $\Gamma_1$).
We mention three reasons  for our interest in $\Gamma_1$:
\begin{itemize}
\item There are only few known examples of lattices in $PU(n,1)$ generated by complex
reflections when $n > 3$ and these are all arithmetic.
The two largest values of $n$ for which an example is known are $13$ and $9$.
Sources for these examples are Deligne-Mostow \cite{DM:monodromy, Mo:generalized, Mo:discontinuous}, 
Thurston \cite{T:shape} and Allcock \cite{A:NC, A:Leech}.
The ``largest" examples found in \cite{Mo:discontinuous} and \cite{T:shape} are identical; it is 
an arithmetic lattice in $PU(9,1)$. This lattice is denoted by
$\Gamma_5$ later in this introduction. A single example in
dimension thirteen ($\Gamma_2$ in our notation) was found in \cite{A:Leech}.
The group $\Gamma_1$ of this article clearly gives a new example in $PU(9,1)$.
\item Allcock's monstrous proposal conjecture states that the fundamental group of 
the ball quotient constructed from $\Gamma_2$ maps onto $(M \wr 2)$ where $M$ is the monster simple group. The arithmetic lattice $\Gamma_2 \subseteq PU(13,1)$ plays a central role in
the monstrous proposal; see \cite{A:monstrous, TB:el, AB:braidlike, AB:26}. Our results  for $\Gamma_1$
have striking similarity with results for $\Gamma_2$ obtained in \cite{TB:el}.
\item The results for $\Gamma_1$ and $\Gamma_2$ (and a few other lattices in $PU(n,1)$)
tie into a general pattern of phenomenon that have analogy
in the theory of Weyl groups; see Theorem \ref{th-1}. Understanding this analogy
may be useful in finding more interesting examples and in studying complex reflection groups in general.
\end{itemize}
Before discussing $\Gamma_1$ in more detail, we want to describe this general
pattern of phenomenon.
Let $\K$ denote one of three real division algebras $\R$, $\C$ or $\mathbb{H}$.
Let $V$ be a $\K$-module with a non-degenerate $\K$-valued Hermitian form
which is either positive definite or Lorentzian (i.e. of signature $(n,1)$). Let $G$ denote the real
Lie group of isometries of $V$. Let $\Gamma$ be a discrete subgroup of $G$
generated by real, complex or quaternionic reflections.
Let $X$ be the projective space over $\K$ if $V$ is positive definite and let $X$ be
the hyperbolic space over $\K$ if $V$ is Lorentzian.
There is a natural $G$-invariant metric on the symmetric space $X$.
The discrete group $\Gamma$ acts properly discontinuously on $X$ by isometries.
The fixed points of a reflection $s \in \Gamma$ is a totally geodesic
$\K$-hypersurface in $X$ called the {\it mirror} of $s$.
Nice generators and relations for $\Gamma$ can often be found as follows:
{\it Choose a suitable point $\tau$ in $X$ that is a point of symmetry of the mirrors
of $\Gamma$ in an appropriate sense. Let $S$ be the set of reflections in $\Gamma$
whose mirrors are closest to $\tau$. In many examples, the reflections in $S$ generate $\Gamma$
and these reflections satisfy nice Coxeter-type relations. }
First, we give some well known examples:

\begin{itemize}
\item Let $\Gamma$ be a Weyl group of $A$-$D$-$E$ type acting on the real vector space $V$
by its natural reflection representation. Let $\tau$ be the line in $V$ containing the
Weyl vector. Then the mirrors of the simple reflections
are exactly the mirrors in $P(V)$ closest to $\tau$. So in this case $S$ is just the
set of simple mirrors. In analogy with this classical case, in all the examples
described below, the reflections in $S$ will be called {\it simple reflections}
and the corresponding mirrors will be called {\it simple mirrors}.
\item Let $\Gamma$ be an irreducible finite complex reflection group acting on 
$V = \C^n$. For most $\Gamma$ including the infinite family $G(de, e, n)$,
there exists a point $\tau$ in $P(V)$ such that the mirrors closest to 
$\tau$ generate $\Gamma$ (see \cite{TB:uggr}, section 3.10 and section 3.11, remark (5)).
For the infinite family $G(de, e, n)$
one can choose $\tau$ such that $S$ is the standard set of generators, given,
for example in \cite{BMR:CR}.
However, for some exceptional $\Gamma$, one can choose a $\tau$
that are analogous to a Weyl vector (in a certain sense, explained in \cite{TB:uggr})
and that yield a set of generators
$S$ different from the standard ones given in \cite{BMR:CR}.
\item The reflection group $\Gamma$ of $\even_{25,1}$ (the unique even self-dual $\Z$-lattice
of signature $(25,1)$) acts on 
the real hyperbolic space $X = \R H^{25}$. Chooses $\tau$
to be a ``Leech cusp", which means that $\tau$ is a line containing a primitive norm zero vector
$\tau_*$ such that $\tau_*^{\bot}/\tau_*$ is isomorphic to the Leech lattice.
Here we are stretching our discussion a little bit since 
$\tau$ is not really a point in $\R H^{25}$ but a point on its boundary.
The mirrors closest to $\tau$ in horocyclic distance are parametrized by
the vectors of the Leech lattice (modulo $\pm 1$). So they are called
Leech mirrors. The reflections in the Leech mirrors generate $\Gamma$.
These generating reflections obey Coxeter relations governed by the Leech lattice, 
leading to Conway's observation: ``the Leech lattice is the Dynkin diagram of the
reflection group of $\even_{25,1}$" \cite{C:Aut26}.
\end{itemize}
Counting the example studied in this article, we now have at least six examples
of complex and quaternionic hyperbolic reflection groups, where similar results hold.
To state these results in an uniform manner, we need some notation.
Let $\cG$, $\cE$ and $\cH$ denote the ring of Gaussian integers, 
the ring of Eisenstein integers and the quaternionic ring of Hurwitz integers
respectively. Let $\cO$ be one of these three rings.
Let $l$ be a nonzero prime in $\cO$ of smallest possible norm.
If $\cO = \cG$ or $\cO = \cH$, we may choose $l = p = (1 + i)$. If $\cO = \cE$, 
we may choose $l = \sqrt{-3}$. 
%So $p = (1+i)$ if $\cO = \cG$ or $\cH$ and $p = \sqrt{-3}$ if $\cO = \cE$.
Let $\cO_{n,1}$ 
denote a $l$-modular Hermitian $\cO$--lattice of signature $(n,1)$,
if such a lattice exists\footnote{This means $\cO_{n,1}$ is a free (right) $\cO$--module 
of rank $(n+1)$ with a $\cO$--valued Hermitian form 
$ \ip{\;}{\;} :\cO_{n,1} \times \cO_{n,1} \to \cO$
of signature $(n,1)$, and $l^{-1} \cO_{n,1}$ 
is equal to the dual lattice of $\cO_{n,1}$.}.
In particular, we define $\cO_{1,1} = \cO e_1 \oplus \cO e_2$ where 
$e_1^2 = e_2^2 = 0$ and $\ip{e_1}{e_2} = \bar{l}$. 
The lattice $\cO_{1,1}$ is the unique $l$-modular Hermitian $\cO$-lattice of 
signature $(1,1)$ and we call it the {\it hyperbolic cell}.
We shall consider the following six lattices:
\begin{equation*}
L_1 = \cG_{9,1},  \;\;
L_2 = \cE_{13,1},\;\;
L_3 = \cH_{7,1}, \;\;
L_4 = \cG_{5,1}, \;\;
L_5 = \cE_{9,1}, \;\;
L_6 = \cH_{5,1}.
\end{equation*}
In each case $L_j$ is the unique $l$-modular Hermitian $\cO$-lattice in its given signature.
Note that $L_{3 + j}$ is a sub-lattice of $L_j$.
%The appropriately scaled real form of these six lattices respectively are
%\begin{equation*}
%4 D_4 \oplus \bigl( \begin{smallmatrix} 0 & 1 \\ 1 & 0 \end{smallmatrix} \bigr), \;
%I \! I_{28,4},\; I \! I_{26,2}, \; I \! I_{28,4},\;
%2 D_4 \oplus \bigl( \begin{smallmatrix} 0 & 1 \\ 1 & 0 \end{smallmatrix} \bigr), \;
%I \! I_{18,2},\; I \! I_{20,4}. 
%\end{equation*}
Let $R(L_j)$ be the (complex or quaternionic) reflection group of $L_j$ and let 
$\Gamma_j = P R(L_j)$ be the image of $R(L_j)$ in $PU(n,1)$.
The mirrors of $\Gamma_j$ are determined by the orthogonal complements
of vectors of minimal positive norm in $L_j$. Since $L_j$ is indefinite, there are infinitely many mirrors.
If $\cO =\cE$, then the reflection group $\Gamma_j$ contains
order $3$ reflections around the mirrors. 
If $\cO = \cG$ or $\cH$, then $\Gamma_j$ contains
order $4$ and order $2$ reflections around the mirrors.
The projectivized reflection group
$\Gamma_j$ acts faithfully on the complex or quaternionic hyperbolic space $X_j$ of appropriate
dimension. The following results hold:
\begin{theorem}
(a) $\Gamma_j$ has finite index in 
$P\op{Aut}(L_j)$. So $\Gamma_j$ is arithmetic.
\par
(b) There is a point $\tau_j$ in $X_j$ such that the set of reflections $S_j$
in the mirrors closest to $\tau_j$ (i.e. the simple mirrors) generate $\Gamma_j$. Further,
$P\op{Aut}(L_j)$ has a finite subgroup $ Q_j$ that acts transitively
on the simple mirrors and $\tau_j$ is the unique point of $X_j$
fixed by $Q_j$. 
\par
(c) The Coxeter relations between the simple reflections $S_j$
are encoded by the edges of a diagram $D_j$ which we call the Dynkin diagram of $\Gamma_j$.
The vertices of $D_j$ correspond to the simple reflections.
The diagram $D_2$ (resp. $D_3$) 
is the incidence graph of $P^2(\F_3)$ (resp.  $P^2(\F_2)$).
The  diagram $D_1$ has $32$ vertices and is 
defined by the incidence relations of $16$ points
and sixteen hyperplanes in $\F_2^4$.
A precise description of $D_1$ is given later in this introduction.
Finally $D_{3 + j}$ is a maximal circuit in $D_j$.
\label{th-1}
\end{theorem}
In each case, $Q_j$ is roughly the automorphism group of the diagram $D_j$.
Theorem \ref{th-1} is the summary of results from a few articles.
For $\Gamma_2, \Gamma_3, \Gamma_5, \Gamma_6$, part (a) is due to Allcock
\cite{A:NC, A:Leech}.  For $\Gamma_4$, theorem \ref{th-1} is due to Goertz \cite{Go:Thesis} (unpublished). 
The rest of the results are due to the author.
In this paper, we prove Theorem \ref{th-1} for $\Gamma_1$.
Theorem \ref{th-1} for $\Gamma_2$, $\Gamma_3$, $\Gamma_5$
follow from the results in \cite{TB:el}, \cite{TB:ql} and section 4.1 of \cite{TB:Thesis} respectively.
The generators and relations for $\Gamma_2$ encoded in the diagram $D_2$
form the basis for Allcock's monstrous proposal conjecture \cite{A:monstrous}.
One of our motivation for studying $\Gamma_1$ in detail is the
close similarity between $\Gamma_1$ and $\Gamma_2$ and our interest in
 $\Gamma_2$ stemming from the monstrous proposal conjecture.
The proofs of part (b), (c) for $\Gamma_6$ have not been written up.
However the proofs for $\Gamma_{3 + j}$ are entirely
similar to the proofs for $\Gamma_j$ and easier. 
A detailed study of the references mentioned and this paper reveal
many more similarities between these reflection groups. 
\par
For the rest of the introduction, we shall focus on  $\Gamma_1$
and give some more details.
To maintain notational consistency with the references cited above, 
we shall drop the subscript and write
$L = L_1$, $\tau = \tau_1$ and so on. 
 We work over the ring
$\cG = \Z[i]$. Let $p = (1 + i)$. 
Our objective is to study the complex reflection group of the unique $p$-modular 
$\cG$--lattice $L = \cG_{9,1}$ of signature $(9,1)$. 
%The uniqueness of such an $L$ follows from the uniqueness of even self-dual $\Z$--lattice of prescribed indefinite signature.
One has 
\begin{equation*}
L 
\simeq 4 D_4^{\cG} \oplus \cG_{1,1}
\simeq \bwg \oplus \cG_{1,1}.
\end{equation*}
where $D_4^{\cG}$ and $\bwg$ are the Gaussian forms of the $D_4$ root lattice 
and the Barnes-Wall lattice respectively.
The projective reflection group $\Gamma$  is a discrete subgroup of
$P U(9,1)$ and acts faithfully by isometries on the complex hyperbolic space 
\begin{equation*}
\B(L) = \lbrace  \C v \colon v \in L \otimes_{\cG} \C, \; \ip{v}{v} < 0 \rbrace \simeq \C H^9.
\end{equation*}
By definition, $R(L)$ is generated by $i$-reflections 
(order $4$ complex reflections) in the norm $2$ vectors of $L$.
By definition, $\Gamma = P R(L)$ is the image of $R(L)$ in $PU(9,1)$.
\begin{theorem}
(a) (See \ref{th-Gamma-is-arithmetic}) $\Gamma$ has finite index in $P \! \op{Aut}(L)$. So $\Gamma$ is arithmetic.
\par
(b) (See \ref{th-13-reflections-generate-R(L)}, \ref{th-32-reflections-generate-R(L)}) $\Gamma$ is generated by thirteen $i$-reflections satisfying the Coxeter relations of the diagram
$X_{3333}$ shown in figure \ref{fig-X3333}.
\par
(c) (See \ref{l-L-from-D}, \ref{th-simple-mirrors}, \ref{t-configuration})
The above generating set of thirteen $i$-reflections can be extended to a set of 
thirty-two $i$-reflections whose mirrors are equidistant from a point $\tau$ in $\B(L)$. 
These $32$ mirrors are precisely the mirrors closest to $\tau$.
A subgroup $Q$ of $P \! \op{Aut}(L)$ isomorphic to $(2^4 \colon (2^3 \colon L_3(2))) \colon 2$
(in \cite{C:ATLAS} notation)
acts transitively on the $32$ mirrors and $\tau$ is the unique point in $\B(L)$ fixed by $Q$.
\end{theorem}
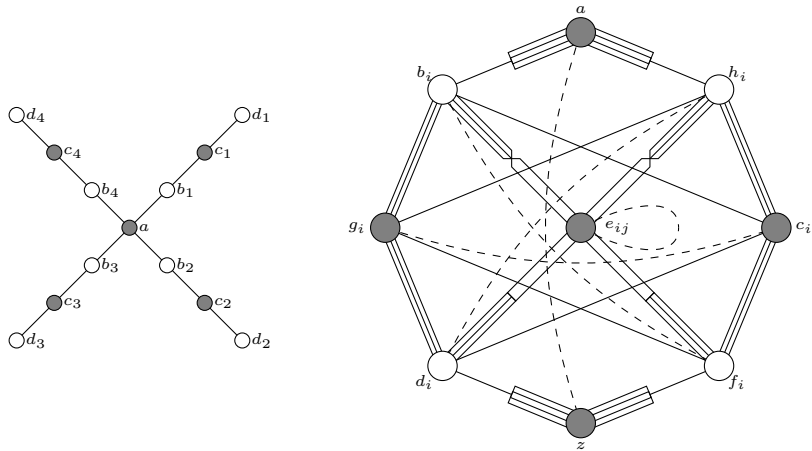
\begin{figure}
\centerline{
 \begin{tikzpicture}[scale=.5]
\begin{scope}[xshift=-6cm]
\draw (-3,-3) -- (3,3);
\draw (3,-3) -- (-3,3); 
\draw [fill = gray] (0, 0) circle [radius = .2];
\foreach \x in {1,...,4} 
{
\draw [rotate =  \x*360/4, fill = white] (1, 1) circle [radius = .2];
\draw [rotate =  \x*360/4, fill = gray] (2, 2) circle [radius = .2];
\draw [rotate =  \x*360/4, fill = white] (3, 3) circle [radius = .2];
 }
 \node [right] at ( 0, 0) {\tiny $a$};
 \node [right] at ( 1, 1) {\tiny $b_1$};
 \node [right] at ( 2, 2) {\tiny $c_1$};
 \node [right] at ( 3, 3) {\tiny $d_1$};
 \node [right] at ( 1,-1) {\tiny $b_2$};
 \node [right] at ( 2,-2) {\tiny $c_2$};
 \node [right] at ( 3,-3) {\tiny $d_2$};
 \node [right] at (-1,-1) {\tiny $b_3$};
 \node [right] at (-2,-2) {\tiny $c_3$};
 \node [right] at (-3,-3) {\tiny $d_3$};
 \node [right] at (-1, 1) {\tiny $b_4$};
 \node [right] at (-2, 2) {\tiny $c_4$};
 \node [right] at (-3, 3) {\tiny $d_4$};
 \end{scope}
 \begin{scope}[xshift=6cm, scale = 1.3]
 \draw (0,4.2) -- (1.485, 3.585) -- (1.3435, 3.2435) -- (0,3.8);
 \draw (0,4.07) -- (1.4390, 3.474) -- (1.390, 3.355) -- (0,3.93);
 \draw (1.414, 3.414) -- (2.828,2.828); 
 \draw (0,4.2) -- (-1.485, 3.585) -- (-1.3435, 3.2435) -- (0,3.8);
 \draw (0,4.07) -- (-1.4390, 3.474) -- (-1.390, 3.355) -- (0,3.93);
 \draw (-1.414, 3.414) -- (-2.828,2.828); 
 \draw (0,-4.2) -- (1.485,-3.585) -- (1.3435, -3.2435) -- (0,-3.8);
 \draw (0,-4.07) -- (1.4390, -3.474) -- (1.390, -3.355) -- (0,-3.93);
 \draw (1.414, -3.414) -- (2.828,-2.828); 
 \draw (0,-4.2) -- (-1.485,-3.585) -- (-1.3435, -3.2435) -- (-0,-3.8);
 \draw (-0,-4.07) -- (-1.4390, -3.474) -- (-1.390, -3.355) -- (-0,-3.93);
 \draw (-1.414, -3.414) -- (-2.828,-2.828); 
\draw (2.828, 2.828) -- (-4,0) ; \draw (-2.828,2.828) -- (4,0); 
 \draw  (2.828,-2.828) -- (-4,0) ; \draw (-2.828,-2.828) -- (4,0);
 \draw ( 4,0) -- ( 2.828, 2.828) [xshift=3pt] ( 4,0) -- ( 2.828, 2.828) [xshift=-6pt] ( 4,0) -- ( 2.828, 2.828);
 \draw ( 4,0) -- ( 2.828,-2.828) [xshift=3pt] ( 4,0) -- ( 2.828,-2.828) [xshift=-6pt] ( 4,0) -- ( 2.828,-2.828);
 \draw (-4,0) -- (-2.828, 2.828) [xshift=3pt] (-4,0) -- (-2.828, 2.828) [xshift=-6pt] (-4,0) -- (-2.828, 2.828);
 \draw (-4,0) -- (-2.828,-2.828) [xshift=3pt] (-4,0) -- (-2.828,-2.828) [xshift=-6pt] (-4,0) -- (-2.828,-2.828);
 \foreach \x in {0,...,1}
 {
 \draw [rotate = 360/8 + \x*90] (0, .12) -- (1.88,  .12) -- (2.12, -.12) -- (4, -.12);  
 \draw [rotate = 360/8 + \x*90] (0,-.12) -- (1.88, -.12) -- (2.12,  .12) -- (4,  .12);
 \draw [rotate = 360/8 + \x*90] (2,0) -- (4,0); 
 \draw [rotate = 5*360/8 + \x*90] (0, .12) -- (2,  .12) -- (2, -.12) -- (4, -.12);  
 \draw [rotate = 5*360/8 + \x*90] (0,-.12) -- (2, -.12) -- (2,  .12) -- (4,  .12);
 \draw [rotate = 5*360/8 + \x*90] (2,0) -- (4,0);  
 }
 \draw [dashed] ( 0,0) to [out=-30, in=-90]  (2,0);
 \draw [dashed] ( 0,0) to [out= 30, in= 90]  (2,0);
 \foreach \x in {1,...,4}
 {
 %\draw [dashed, rotate = \x*360/8] (0,4) arc [radius=4.3, start angle=70, end angle= 290];
 \draw [dashed, rotate = \x*360/8] ( 4,0) to [out=162, in=18]  (-4,0);
 }
 \foreach \x in {1,...,4} 
{
 \draw [fill = gray, rotate=\x*360/4] (4,0) circle [radius=.3cm];
 \draw [fill = white, rotate=\x*360/4] (2.828,2.828) circle [radius=.3cm];
 }
 \draw [fill = gray] (0,0) circle [radius=.3cm];
\node [right] at ( .3,0) {\tiny $e_{i j}$};
\node [above] at ( 0, 4.2) {\tiny $a$};
\node [below] at ( 0, -4.2) {\tiny $z$};
\node [above left] at (-2.818, 2.818) {\tiny $b_i$};
\node [right] at (4.2, 0) {\tiny $c_i$};
\node [below left] at (-2.818, -2.818) {\tiny $d_i$};
\node [below right] at (2.818, -2.818) {\tiny $f_i$};
\node [left] at (-4.2,0) {\tiny $g_i$};
\node [above right] at (2.818, 2.818) {\tiny $h_i$};
   \end{scope}
 \end{tikzpicture}
 }
 \caption{The $X_{3333}$ diagram on the left. A shorthand drawing of the
 $32$ node diagram $D$ on the right.
 In the $32$ node diagram $1 \leq i < j \leq 4$. So the node $c_i$ (resp. $e_{i j}$)
 stands for four (resp. six)
 nodes. 
 A solid (resp. dotted) edge between two vertices $x$ and $y$ indicates the 
 relation $ x y x = y x y$ (resp. $x y x y = y x y x $).
 No edge between $x$ and $y$ means $x y = y x$.
 The following shorthands are used:
The edge between $a$ and $b_i$ means
 that $a$ and $b_i$ are connected for all $i$.
 A single (resp. triple) edge between $b_i$ and $c_i$ (resp. $g_i$)
 means that $b_i$ and $c_j$ (resp. $g_j$) are connected if $i = j$ (resp. $i \neq j$).
 Notice that there are two kinds of edges out of $e_{i j}$. 
 The edge between $e_{i j}$ and $d_i$ (resp. $b_i$) means that $e_{i j}$ is connected
 to $d_k$ (resp. $b_k$) if $k \in \lbrace i , j \rbrace$ (resp. $k \notin \lbrace i , j \rbrace$).
 Finally, the dotted loop joining $e_{i j}$ to itself means that $e_{i j}$ is joined to
 $e_{ k l}$ if $ \lbrace i , j \rbrace \cap \lbrace k , l  \rbrace = \emptyset$.
 }
 \label{fig-X3333}
 \end{figure}
The configuration of the $32$ mirrors  closest to $\tau$
has appealing symmetry related to the geometry of the finite vector space
$\F_2^4$. To describe this, fix a point $a \in \F_2^4$.
For $u \in \F_2^4$, let $t_u : \F_2^4 \to \F_2^4$ be the translation $t_u(v) = u + v$.
Let $\cK_0$ be the set of hyperplanes in
$\F_2^4$ that do not contain $a$. Let $\cK$ be the set of translates of the hyperplanes
in $\cK_0$. So $\cK$ consists of $8$ homogeneous and $8$ affine hyperplanes in $\F_2^4$.
Let $D = \F_2^4 \cup \cK$. Note that each $t_u$ permutes $\F_2^4$ and permutes $\cK$
and thus defines a permutation of $D$.
The thirty-two $i$-reflections closest to $\tau$ can be labeled by $D$
such that the relations among these $i$-reflections are dictated by the configuration $D$. 
More precisely, let $d , d' \in D$ and let $R, R'$ be the corresponding $i$-reflections.
\begin{itemize}
\item If $\lbrace d, d' \rbrace$ is an incident pair of point and hyperplane, 
then $R R'  R = R' R R'$. 
This is denoted in the diagram $D$ (fig. \ref{fig-X3333}) by a solid edge joining $d$ to $d'$.
\item If  $d' = t_a(d)$,
then $R R' R R' =  R' R R' R$.
This is denoted in the diagram $D$ by a dotted edge joining $d$ to $d'$. 
\item Otherwise, $R R' = R' R$. 
\end{itemize}
We picture $D$ as a graph with two kinds of edges.
The subgroup $2^3 \colon L_3(2)$ in $Q$ is the stabilizer of $a$ in $L_4(2)$,
the $2^4$ corresponds to the translation action of $\F_2^4$ on itself,
and the extra $\Z/2$ is a symmetry that interchanges the
points in $\F_2^4$ and hyperplanes in $\cK$ (see \ref{t-configuration} for details).
The group $Q$
acts on the set $D$ preserving both kind of edges.
 We may think of $D$ as the Dynkin diagram for $R(L)$ and $Q$ as the group
 of diagram automorphisms. 
 \par
The proof showing that $\Gamma = P R(\cG_{9,1})$ is arithmetic is similar to the proof for 
$\Gamma_2 = P R(\cE_{13,1})$ in \cite{A:NC}
which in turn is adapted from an argument in \cite{C:Aut26}. The statements and proofs in this article
often closely parallel those in \cite{A:NC, TB:el, TB:ql}. We shall refrain from
mentioning them at every step, but a  couple of remarks comparing $\cG_{9,1}$ and $\cE_{13,1}$
are worthwhile.
Below $\op{Leech}^{\cE}$ denotes the complex Leech lattice (studied in detail
in \cite{W:complex}) scaled to have minimal norm $6$.
\par
The description 
$\cG_{9,1} 
 \simeq \bwg \oplus \cG_{1,1}$
is crucial in our proofs just like the description
$\cE_{13,1}
\simeq \op{Leech}^{\cE} \oplus \cE_{1,1}$
is crucial in the proofs in \cite{A:Leech, TB:el}.
This is because we use two key properties that are shared by the 
(real) Barnes-Wall and the (real) Leech lattice: they contain no norm $2$ vectors
and  they are very dense,
in fact the densest lattices known in the respective dimensions.
The absence of norm $2$ vector in $\bwg$ provides us with a
``Barnes-Wall cusp" at the boundary of $\B(L)$ such that no mirror passes through it.
These cusps play a key role in our arguments.
\par
 The results on 
$\cE_{13,1} $
use the fact that the covering radius of the Leech lattice is $\sqrt{2}$. 
In our results on
$\cG_{9,1}$,
the density of the Barnes-Wall lattice is used via lemma 6.11 of \cite{A:NC} which is a special 
case of results in \cite{REB:llll}. This lemma gives
a covering of the underlying real vector space of the Barnes-Wall lattice using balls of two sizes.
It is curious to note that this lemma does not use the covering radius of the Barnes-Wall lattice.
Rather, it uses the covering radius of the Leech lattice and an embedding of the
Barnes-Wall lattice in the Leech lattice.
Just like the results in \cite{A:NC, TB:el, AB:braidlike}, the results of this paper would all fail
to hold if the covering radius of the
Leech lattice was any bigger than $\sqrt{2}$. 
%
%
%*******************************************************************************************************
%
\section{Preliminaries}
\begin{definition}[Gaussian lattices]
Let $\cG = \Z[i]$ be the ring of Gaussian integers. Let $p = (1 + i)$. 
Let $K$ be a free $\cG$-module of finite rank with a $\Q[i]$-valued Hermitian
form $\ip{\;}{\;} : K \times K \to \Q[i]$. 
Hermitian forms are always assumed to be linear in second variable. 
We shall always identify $K$ inside the $\Q[i]$-vector space $K \otimes_{\cG} \Q[i]$
and further, inside the complex vector space $K \otimes_{\cG} \C$. 
The Hermitian form linearly extends to these vector spaces.
For $v \in K \otimes_{\cG} \C$, write $v^2 = \ip{v}{v}$.
We say that $v^2$ is the {\it norm} of $v$.
A nonzero vector of norm zero is called a {\it null vector}. 
\par
Let $A \subseteq K \otimes_{\cG} \C$ and $m \in \R$.
It will be convenient to use the notation:
\begin{equation*}
A(m) = \lbrace a \in A \colon a^2 = m \rbrace
\text{\; and \;} 
A(\leq m) = \lbrace a \in A \colon a^2 \leq m \rbrace.
\end{equation*}
%Also, for $r, n \in \Z$, $n \geq 1$,  we shall write $A( r \bmod n) = \cup_{ k \in \Z} A(r + n k)$.
Let $A^{\bot} = \lbrace  v \in K \otimes_{\cG} \C \colon \ip{v}{a} = 0 \; \forall a \in A \rbrace$.
The {\it radical} of $K$ is defined as $\op{rad}(K) = K^{\bot} \cap K$.
The Hermitian form is nonsingular
if $\op{rad}(K) = 0$.
If $\op{rad}(K) = 0$, then $K$ is called a {\it $\cG$-lattice} or a {\it Gaussian lattice}. 
If $\op{rad}(K) \neq 0$, then $K$ is  called a {\it singular $\cG$-lattice}. 
Say that $K$ is {\it integral} if the Hermitian form takes values in $\cG$.
Say that $K$ is {\it Lorentzian} if it has signature
$(n,1)$. 
\par
Let $K$ be a $\cG$--lattice. Define 
$K^{\vee} = \lbrace v \in K \otimes_{\cG} \Q[i] \colon \ip{v}{K} \subseteq \cG \rbrace$.
Then $K^{\vee}$ is a $\cG$-lattice called the {\it dual lattice} of $K$.
Note that $K$ is integral if and only if $K \subseteq K^{\vee}$.
%If $K$ is integral, the finite abelian group $K^{\vee}/K$ is called the {\it discriminant group} of $K$.
Let $l$ be a prime in $\cG$.
A $\cG$-lattice $K$ is called {\it $l$-modular}
if $K^{\vee} = l^{-1} K$. 
Clearly if $K_1$ and $K_2$ are $l$-modular, then so is $K_1 \oplus K_2$.
Let $K_{\Z}$ denote the underlying $\Z$--lattice of $K$.
This means that $K_{\Z}$ is the underlying $\Z$-module of $K$ with the
 bilinear form $\op{Re} \ip{\;}{\;}$.
Note that $K$ is an integral $\cG$-lattice if and only if $K_{\Z}$ is an integral
$\Z$-lattice and $(K^{\vee})_{\Z} = (K_{\Z})^{\vee}$.
\end{definition}
\begin{topic}{\bf The $D_4^{\cG}$ lattice:}
 Let $D_{2n}^{\cG}$ be the sub-lattice of $\cG^n$
consisting of all $(x_1, \dotsb, x_n)$ in $\cG^n $ such that 
$(x_1 + \dotsb + x_n) \equiv 0 \bmod p$ with the standard positive
definite Hermitian form
\begin{equation*}
\ip{x}{x'} = \bar{x}_1x_1' + \dotsb +  \bar{x}_n x_n'. 
\end{equation*}
The underlying $\Z$-module of $D_{2n}^{\cG}$ with 
the inner product $\op{Re} \ip{x}{y}$ is the root lattice $D_{2n}$, where 
\begin{align*}
D_{n} &= \lbrace (x_1, \dotsb, x_n) \in \Z^n \colon (x_1 + \dotsb + x_n) \equiv 0 \bmod 2 \rbrace.
\end{align*}
Note  that  $D_4^{\cG}$ is $p$-modular but 
 $D_n^{\cG}$ is not $p$-modular for $n > 2$.
 A $\cG$-basis for $D_4^{\cG}$ is $v_1 = (1,1)$ and $v_2 = (0,\bar{p})$.
The discriminant group of $D_4^{\cG}$ is
 $(D_4^{\cG})^{\vee} / D_4^{\cG} = p^{-1}D_4^{\cG}/ D_4^{\cG}  \simeq (\Z/2)^2$.
Coset representatives for $ p^{-1}D_4^{\cG}/ D_4^{\cG}$ can be chosen to be 
 $\lbrace (0,0), v_1/\bar{p},  v_2/\bar{p}, (v_2 - v_1)/\bar{p} \rbrace$.
  \end{topic}
 \begin{topic}{\bf The Barnes-Wall lattice over Hurwitz quaternions and Gaussian integers: }
Express real quaternions in the form $(x + y j)$ where $x, y \in \C$.
The ring $\mathcal{H}$ of Hurwitz integers consists of all $(x + y j)$ such that 
$(x , y) \in  \cG^2$ or $(x + \tfrac{p}{2}, y + \tfrac{p}{2} ) \in \cG^2 $. 
Note that $\cG$ is a subring of $\mathcal{H}$.
The map  $(x + y j) \mapsto (x,y)$
defines an isomorphism $\mathcal{H} \simeq p^{-1} D_4$ as 
$\cG$-modules.
Define
\begin{align*}
\bwg
& = \lbrace (x_1,\dotsb, x_4) \colon 
 x_j \in p^{-1} D_4^{\cG},  \; x_j \equiv x_k \bmod D_4^{\cG} \; \forall j, k, 
 \; \; \sum_{j} x_j \in p D_4^{\cG}
   \rbrace
\end{align*}
Allcock \cite{A:NC} describes a four dimensional Hurwitz lattice whose real form (appropriately scaled)
is the Barnes-Wall lattice.
It is immediate that $\bwg$, as defined above,
is the Gaussian form of this Hurwitz Lattice with 
the norms scaled by $\sqrt{2}$. So the real form of $\bwg$
is the usual rank sixteen Barnes-Wall lattice of minimum norm $4$.
For more information on the Barnes-Wall lattice see \cite{CS:SPLAG, VS:genus, NRS:simple}.
The sixteen dimensional Barnes-Wall lattice is part of a family of lattices studied widely in the coding
theory literature: see \cite{GP:list, NRS:simple} and the references in there. 
An alternative quick definition of the $\cG$-lattice $\bwg$ is the $\cG$-span of 
the rows of $\bigl( \begin{smallmatrix} 1 & 1 \\ 0 & p \end{smallmatrix} \bigr)^{\otimes 4}$ \cite{GP:list}.
It is straightforward to verify the equivalence of the two definitions.
\end{topic}
\begin{topic}{\bf A common over-lattice of $4D_4^{\cG}$ and $\bwg$:}
 Let $M_{16}^{\cG}$ be the rank $8$ Gaussian lattice
 \begin{align*}
M_{16}^{\cG}  
=  \lbrace (x_1, x_2, x_3, x_4) : x_j \in  p^{-1} D_4^{\cG} \colon x_j \equiv x_k \bmod D_4^{\cG} 
 \; \forall \;  j, k \rbrace .
 \end{align*}
Then $M_{16}^{\cG}$ is an integral Gaussian lattice
 of minimum norm $2$ that contains both $4D_4^{\cG}$ and $\bwg$. 
 It is easy to verify the following inclusions among lattices with the indices
indicated next to the edges:
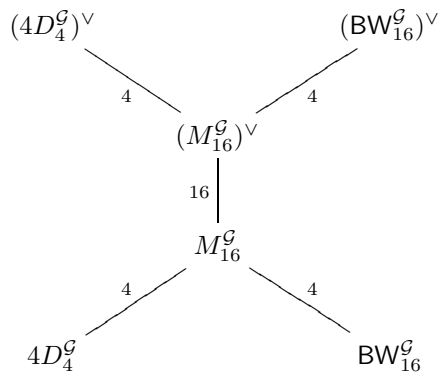
\begin{equation*}
\xymatrix{
(4D_4^{\cG})^{\vee} \ar@{-}[dr]_{4} & & (\bwg)^{\vee}  \ar@{-}[dl]^4\\
 & (M_{16}^{\cG})^{\vee} \ar@{-}[d]_{16} & \\
 &  M_{16}^{\cG}  & \\
4D_4^{\cG} \ar@{-}[ur]^{4} & & \bwg \ar@{-}[ul]_{4}
}
\end{equation*}
In particular $\abs{(\bwg)^{\vee} / \bwg} = 2^8$.
On the other hand, from the definition of $\bwg$ one verifies that
$ (\bwg)^{\vee} \supseteq p^{-1} \bwg$.
So 
\begin{equation*}
2^8 = \abs{ (\bwg)^{\vee}/\bwg } \geq 
\abs{ p^{-1} \bwg/\bwg } = \abs{ p^{-1} \cG^8/ \cG^8} = 2^8
\end{equation*}
It follows that equality must hold everywhere and that
$ (\bwg)^{\vee} = p^{-1} \bwg$. In particular, all vectors in $\bwg$ have even norm.
\par
One crucial property of $\bwg$ is that it does not have any norm $2$ vector, 
so it has minimum norm $4$.
This can be seen quickly from our definition as follows:
Note that  $ p^{-1} D_4^{\cG}$ has minimum norm $1$. 
Take $x \in \bwg(2)$. Write $x = (x_1, x_2, x_3, x_4)$
with each $x_j \in p^{-1} D_4^{\cG}$.
Then  $\sum_{j = 1}^4 x_j^2  = 2$ implies that
at least $2$ of the $x_j$'s must be $0$. Since the $x_j$'s are all congruent modulo $D_4^{\cG}$, it
follows that $x_j \in D_4^{\cG}$ for all $j$. Since $D_4^{\cG}$ has minimum norm $2$, it follows
that there exists a $k \in \lbrace 1, 2, 3, 4 \rbrace$ such that $x_j = 0$ for $j \neq k$.
But now $x_k =\sum_j x_j  \in p D_4^{\cG}$,
so $x^2 = x_k^2$ has norm at least $4$, which is a contradiction.
\end{topic}
The lemma below is taken from \cite{A:NC} and is a special case of the results in \cite{REB:llll} where
it was used to find interesting reflection groups in real hyperbolic space $\R H^{17}$
(see theorem 3.1 of \cite{REB:llll} and the examples following theorem 3.3 on page 232).
As mentioned in the introduction,
the proof depends on the covering radius of the Leech lattice.
We quote it below for convenience:
\begin{lemma}[same as lemma 6.11 of \cite{A:NC}]
\label{l-covering-lemma}
The real vector space $\bwg \otimes_{\Z} \R$ is covered by the closed balls of radius 
$\sqrt{2}$ around the vectors of $\bwg$ together with closed balls of radius $1$ around the vectors
$p^{-1} v$ with $v \in \bwg$ and  $v^2 \equiv 2 \bmod 4$.
\end{lemma}

\begin{definition}[roots, reflection groups]
Let $V$ be  a complex vector space with a
Hermitian form $\ip{\;}{\;}$.
Let $v \in V$ be a vector of nonzero norm. 
A {\it complex reflection} $R$ in $v$ is a linear automorphism of $V$
of finite order that point-wise fixes $v^{\bot}$. If $R$ has order $n$, then
it follows that $R(v)= \xi v$ where $\xi$ is a primitive $n$-th root of unity.
We shall write $R = R_v^{\xi}$. One has
\begin{equation*}
R_v^{\xi}(x) = x - (1 - \xi) \frac{\ip{v}{x}}{ v^{2}} v. 
\end{equation*}
Let $K$ be a $\cG$--lattice. A {\it root} of $K$ means a primitive vector $v$ of $K$
of positive norm such that
$R_v^{\xi} \in \op{Aut}(K)$ for some non-trivial root of unity $\xi$. 
The {\it reflection group} of $K$, denoted $R(K)$, is the
subgroup of $\op{Aut}(K)$ generated by the reflections in the roots
of $K$. Write $R_v = R_v^{i}$. 
Let $s$ and $t$ be two norm $2$ vectors of $K$. One verifies that
$R_s$ and $R_t$ commutes (reps. braids) if  $\ip{s}{t} = 0$ 
(resp. $\abs{\ip{s}{t} } = \sqrt{2}$). Further, if $\abs{\ip{s}{t}} = 2$, then
$ i R_s R_t R_s (t) = t$, and hence
$R_s R_t R_s R_t = R_t R_s R_t R_s$.
 \end{definition}
\begin{topic}{\bf Reflection groups of $p$-modular lattices: }
Assume $K$ is a $p$-modular $\cG$-lattice. Then the minimal norm of
$K$ is at least $2$.
From the formula for complex reflection we find that $K(2)$  is the set of 
roots of $K$ and the reflections in $R(K)$ are precisely the order $4$
and order $2$ reflections in these roots. 
\end{topic}
\begin{topic}{\bf The reflection group of $D_4^{\cG}$:} 
The lattice $D_4^{\cG}$ has six roots counted up to roots of unity. These are
$(p,0)$, $(0, p)$, $(1, i^r)$, for $r = 0, 1, 2, 3$. If $s$ and $t$ are any two
non-proportional and non-orthogonal roots of $D_4^{\cG}$, then 
the finite complex reflection group $R(D_4^{\cG})$ is generated by the $i$-reflections
$R_s$ and $R_t$. These two reflections obey the braiding relation
$R_s R_t R_r= R_t R_s R_t$. The relations $R_s^4 = R_t^4 = 1$ and the braiding relation
are sufficient to give a presentation of
$R(D_4^{\cG})$. 
This group is called $G_8$ in the table of finite complex reflection groups given in 
\cite{BMR:CR}.
\end{topic}
\begin{topic}{\bf Complex hyperbolic space: }
\label{t-complex-hyperbolic-space}
Let $V$ be a complex vector space of dimension $(n+1)$ with the standard non-degenerate Hermitian
form of signature $(n,1)$. Let $\B(V)$ be the set of one dimensional negative definite subspaces
of $V$. This is an open subset of the projective space $P(V)$. If $L$ is a Lorentzian $\cG$--lattice,
we write $\B(L) = \B(L \otimes_{\cG} \C)$.
\par
The group of isometries $V$ is $U(n,1)$ and $\B(V)$ is a concrete model for
the corresponding Hermitian symmetric space,
sometimes called the complex hyperbolic space
and denoted by $\C H^n$.
 Up to scaling, there is
a unique $PU(n,1)$ invariant metric on $\C H^n$ called the Bergman metric.
We shall only need some facts about the associated distance function.
\par
A negative norm vector $v$ in $V$ determines a point $\C v$ in $\B(V)$.
A positive norm vector $r$ in $V$ determines a totally geodesic hyperplane $\B(r^{\bot})$
in $\B(V)$. 
For simplicity we shall write $v$ instead of $\C v$ and $r^{\bot}$
instead of $\B(r^{\bot})$ when there is no chance of confusion.
\par
Let $u , v$ be two negative norm vectors in $V$. The distance between the corresponding
points in the complex hyperbolic space is
\begin{equation*}
d( u, v) = \cosh^{-1} \sqrt{ \tfrac{ \abs{\ip{u}{v}}^2 }{ u^2 v^2 } }.
\end{equation*}
Let $r$ be a negative norm vector in $V$. Then
\begin{equation*}
d(r^{\bot}, v) = \sinh^{-1} \sqrt{ \tfrac{ \abs{\ip{r}{v}}^2  }{- r^2 v^2 } }.
\end{equation*}
Let $r, s$ be two negative norm vectors in $V$. If $\C r + \C s$ is positive definite then the hyperplanes
$\B(r^{\bot})$ and $\B(s^{\bot})$ meet in $\B(L)$. Otherwise, one has
\begin{equation*}
d(r^{\bot}, s^{\bot})=  \cosh^{-1} \sqrt{ \tfrac{ \abs{\ip{r}{s} }^2 }{ r^2 s^2 } }.
\end{equation*}
Our distance function differs from the ones in \cite{G:GHG} by a factor of $2$. 
This is not an issue because we only use these formulas only to compare distances.
Let $v \in V$ be a positive norm vector. Let $\xi$ be a primitive $k$-th root of unity
for some $k > 1$.
Unless it is necessary, we shall not distinguish between 
 $R_v^{\xi}$ and its image in $PU(n,1)$ 
 and refer to either as a $\xi$-reflection in $s$. This reflection is an isometry of $\B(V)$
 that point-wise fixes the totally geodesic hyperplane $v^{\bot}$ (called the {\it mirror} of reflection) and
 acts as anti-clockwise rotation of angle $2 \pi/k$ in the normal bundle to the mirror.
\end{topic}
\begin{definition}[horocyclic distance] 
\label{def-horocyclic-distance}
Let $V$ be as in \ref{t-complex-hyperbolic-space}.
Let $z$ be a null vector in $V$.
 If $v$ is a negative norm vector in $V$, define
\begin{equation*}
d_{z}(v)  =\tfrac{1}{2} \log (\abs{\ip{z}{v}}^2/(-v^2)).
\end{equation*}
Note that $d_z$ determines a function on the complex hyperbolic space $\bb{B}(V)$.
We denote this function also by $d_z$.	
The null vector $z$ determines a point $\C z$ in the boundary $\partial \B(V)$ of $\B(V)$.
As before, we write $z$ instead of $\C z$ if there is no chance of confusion.
We say that $d_z(v)$ is the {\it horocyclic distance} between $z$ and $v$.
This terminology is justified by the lemma below. We include a proof
because we could not find a convenient reference.
\end{definition}
\begin{lemma}[ideal triangle inequality]
(a) Let $x, y$ be negative norm vectors in $V$. Then one has $\abs{d_z(x) - d_z(y)} \leq d(x, y)$. 
(b) The equality $d_z(x) - d_z(y) = d(x, y)$ holds if and only if $y$ lies on the
geodesic ray joining $x$ and $z$.
\label{l-iti}
\end{lemma}
\begin{proof}
(a) 
Let $\alpha = \ip{z}{x} $, $\beta = \ip{y}{z}$ and $\gamma = \ip{x}{y}$.
By changing $x, y$ by units if necessary, we may assume, without loss, that,
$\abs{x}^2 = \abs{y}^2 = -1$.
If $z, x, y$ are linearly independent then their span has signature $(2,1)$, so
$\op{det}(\op{gram}(z, x, y)) < 0$, otherwise $\op{det}(\op{gram}(z, x, y)) = 0$. So we have
\begin{align*}
0 
 \geq \op{det}(\op{gram}(z, x, y)) 
= \op{det} \Bigl(  \begin{smallmatrix}  0 & \alpha & \bar{\beta} \\ \bar{\alpha} & -1 & \gamma \\ \beta & \bar{\gamma} & -1 
\end{smallmatrix}  \Bigr) 
&= \abs{\alpha}^2 + \abs{\beta}^2 +2 \op{Re}(\alpha \beta \gamma) \\
&\geq \abs{\alpha}^2 + \abs{\beta}^2 -2 \abs{\alpha \beta \gamma}. 
\end{align*}
It follows that 
\begin{align*}
\cosh d(x,y)  
= \abs{\gamma} 
 \geq  \tfrac{1}{2} ( \tfrac{\abs{\alpha}}{\abs{\beta}} + \tfrac{\abs{\beta}}{\abs{\alpha}} ) 
&= \tfrac{1}{2} ( e^{d_{z}( x) - d_{z}( y)}  +  e^{d_{z}( y) - d_{z}( x)} )  \\
&= \cosh \abs{ d_{z}( x) - d_{z}( y) } .
\end{align*}
Since $( t \mapsto \cosh t )$ is strictly increasing for $ t \in [0, \infty)$, part (a) follows.
\par
(b) Suppose $y$ lies on the geodesic ray joining $z$ and $x$,
Then $z, x, y$ are linearly dependent. So the calculation in part (a) show that
$d(x,y) = \abs{d_{z}(x) - d_{z}( y)}$. 
Now, without loss, assume $\ip{z}{x}$ is a negative real number and $\abs{x}^2 = -1$.
If $y$ is on the geodesic ray joining $x$ and $z$, then
$\C y = \C(x + t z)$ for some $t \geq 0$. 
So 
\begin{equation*}
e^{2 d_{z}(y)} = \tfrac{\abs{\ip{x}{z}}^2 }{- (x + t z)^2 } 
=  \tfrac{\abs{\ip{x}{z}}^2 }{1-  2t \ip{x}{z} }  
< \abs{\ip{x}{z}}^2 = e^{2 d_{z}(x)}.
\end{equation*}
So $d_{z}( x) > d_{z}( y)$ and hence $d(x,y) = d_{z}( x) - d_{z}( y)$.
The other implication follows from uniqueness of geodesic
which is a consequence of negative curvature. We shall skip the details since we do
not need this for our application.
 \end{proof}
\begin{definition}[Horoballs] Let $z$ be a null vector in $V$.
Let $c$ be a positive real number. A subset $B \subseteq \B(V)$ of the form
$B = \lbrace v \colon d_z(v) < c \rbrace$ is called an open {\it horoball} around $z$.
Similarly define a closed horoballs. The boundary of a horoball around $z$ is called a 
{\it horosphere} around $z$.
Pick $v \in \bb{B}(V) -  B$.
Let $p$ be the point where the geodesic ray joining $v$ and $z$ intersects $\partial B$.
Then, one verifies that $p$ is the unique point of the closed horoball $\bar{B}$ that is closest to $v$, that is
\begin{equation*}
d(p, v) =d(B, v).
\end{equation*}
In other words, $p$ is the {\it projection} of $v$ on $B$.
Lemma \ref{l-iti} implies that
\begin{equation*}
d_z(v)   = c + d(B, v) 
\end{equation*}
Let $\zeta \in \partial \bb{B}(V)$ be the point determined by $z$.
We shall say that $v_1$ is closer to $\zeta$ than $v_2$ in horocyclic distance
if and only if $d_z(v_1) < d_z(v_2)$. If we scale $z$, then both sides
of the inequality gets multiplied by the same positive factor; so
this notion does not depend on the
choice of $z$. Another way to see this is to note that
$v_1$ is closer to $\zeta$ than $v_2$ if and only if 
$v_1$ is closer to $B$ than $v_2$, where $B$ is any small horoball around $\zeta$
that misses $v_1$ and $v_2$.
\end{definition}
\section{Reflection groups of $p$-modular $\cG$-lattices: height reduction}
In this section we prove some results about the reflection group of a general $p$-modular
Lorentzian $\cG$-lattice $L$. A null vector $z \in L$ or the point of $\B(L)$ determined by
$z$ is called a {\it cusp} (of $R(L)$).
Our first goal is to prove some lemmas that are useful for finding 
sets of mirrors close to a cusp $z$ such the reflections in them
 generate $R(L)$. Formally, these results are of the following sort:
\begin{lemma}
Let $G$ be a group of  isometries of a metric space $X$. Let $\mathcal{H}$
be $G$-stable collection subsets of $X$ and $A \subseteq X$ such that
$\lbrace d(A, H) \colon H \in \mathcal{H} \rbrace$ is a discrete subset of $[0, \infty)$.
Let $d_0 \in [0, \infty)$. Assume that for all $H \in \cH$ with $d(A, H) > d_0$,
there exists $g \in G$ such that $d( A, g H) < d( A, H)$.
Then $\lbrace H \in \cH \colon d(A, H) \leq d_0 \rbrace$ meets every $G$-orbit in $\cH$.
\end{lemma}
In this situation we say that $d(A, H)$ (or some suitable
increasing function of it) is the {\it height} of $H$ (with respect to $A$).
The proof is an obvious induction on height. 
We call these height reduction arguments.
In our application, $G$ will be some subgroup of $R(L)$, $X = \B(L)$, 
and $A$ will be either a point in $\B(L)$ or a small horoball around some
cusp of $L$. The collection $\cH$ will be either the set of mirrors
of $R(L)$ or a suitable collection of horoballs around the cusps of $L$.
\par
Our second goal of this section is to introduce a discrete Heisenberg group
$\T$ sitting in the stabilizer of a cusp in $\op{Aut}(L)$ and show that the
reflection group $R(L)$ contains a finite index subgroup of $\T$.
\begin{topic}{\bf Notation:}
\label{t-q}
For this section, let $\Lambda$ denote a $p$-modular positive definite $\cG$--lattice
of rank $n$.
Let  $\Lambda(r \bmod 4) = \lbrace \lambda \in \Lambda \colon \lambda^2 \equiv r \bmod 4 \rbrace$.
 Since the underlying $\Z$-lattice of $\Lambda$
is even, $\lambda \mapsto \tfrac{1}{2} \lambda^2 \bmod 2$ is a homomorphism
$ \Lambda \to \Z/2$.
The kernel of this homomorphism is $\Lambda(0 \bmod 4)$  and 
 the complement of the kernel is $\Lambda(2 \bmod 4)$.
If $a \in p \cG$, then
the homomorphism $\Lambda \to \Z/2$ factors through $ \Lambda/a$.
In other words, all the elements in a coset in $\Lambda/ a$
either have norm $0 \bmod 4$ or have norm $2 \bmod 4$. 
\end{topic}
\begin{topic}{\bf $p$-modular Gaussian Lorentzian lattice: }
Let $L = \Lambda \oplus \cell$.
Since $\cell$ is $p$-modular, so is $L$. Note that all vectors of $L$ have even norm.
 Vectors of $L$ will be written in the form
  $(\sigma ; m,n)$ where
 $\sigma \in \Lambda$ and $m,n \in \cG$. 
 The Hermitian form on $L$ is given by
 \begin{align*}
 \bigip{ (\sigma; m,n) }{(\sigma'; m', n') } 
% & = \ip{\sigma}{\sigma'} 
% + (\bar{m}, \bar{n}) \begin{pmatrix} 0 & \bar{p} \\ p & 0 \end{pmatrix} 
% \begin{pmatrix}
% m' \\ n'
% \end{pmatrix} \\
 & = \ip{\sigma}{\sigma'} + \bar{m} \bar{p} n' + \bar{n} p m'.
 \end{align*}
 In particular,
 \begin{equation*}
 (\sigma; m,n)^2 = \sigma^2 + 2 \op{Re}( \bar{m} \bar{p} n).
 \end{equation*}
Let
\begin{equation*}
\rho = (0; 0, 1).
\end{equation*}
This norm zero vector plays a special role throughout.
Note that
\begin{equation*}
\ip{\rho}{(\sigma; m,n)} = p m.
\end{equation*}
If $s \in L(N) - \rho^{\bot}$, then one can  
write $s$ in the form
\begin{equation}
s = \Bigl( \sigma; m, p \bar{m}^{-1} \bigl(  \tfrac{N - \sigma^2}{4} + \nu \bigr) \Bigr),
\label{eq-s-long}
\end{equation}
where $\nu \in \op{Im}(\C)$ is chosen so that the last coordinate of $s$ lies in $\cG$.
We define the {\it height} of a primitive lattice vector $s$ (with respect to $\rho$)
to be
\begin{equation*}
\op{ht}(s) = \begin{cases} \abs{ \ip{s}{\rho}}^2/ \abs{s^2} & \text{\; if  \;} s^2 \neq 0,\\
 \abs{ \ip{s}{\rho}}^2 & \text{\; if  \;} s^2 = 0.
 \end{cases}
\end{equation*}
Given $s \in L(N)$ and $s' \in L(N')$ written in the form 
 \eqref{eq-s-long}, we record an useful formula for their inner product which is verified by direct calculation:
 \begin{equation}
 \op{Re} \ip{ \tfrac{s}{m}}{ \tfrac{s'}{m'}} 
 =  \tfrac{N}{ 2 \abs{m}^2} + \tfrac{N'}{ 2 \abs{m'}^2} 
 - \tfrac{1}{2} \bigl(  \tfrac{\sigma}{m} - \tfrac{\sigma'}{m'} \bigr)^2, 
 \text{\;\;} 
\op{Im} \ip{ \tfrac{s}{m}}{ \tfrac{s'}{m'}}  = 
 \op{Im} \ip{ \tfrac{\sigma}{m}}{ \tfrac{\sigma'}{m'}} +  \tfrac{ 2\nu'}{\abs{m'}^2} - \tfrac{2\nu}{\abs{m}^2} .
 \label{eq-complete-square}
 \end{equation}
\end{topic}
\begin{topic}{\bf The roots of $L$ near the cusp $\rho$:}
\label{l-roots-in-shells-1-2}
The roots of $L$ are the vectors of minimum norm $2$. Let $s \in L(2)$ be a root.
As in \eqref{eq-s-long}, we write
\begin{equation*}
s = ( \sigma; m, p \bar{m}^{-1}  ( \tfrac{1}{2}(1 -\tfrac{\sigma^2}{2}) + \nu )).
\end{equation*}
Note that $\op{ht}(s) = \abs{m}^2$.
The roots having height $1, 2, 4, \dotsb$ are called the roots in the  first shell,
second shell, third shell and so on. 
Mirror of a $j$-th first shell root is called a $j$-th shell mirror 
Among the mirrors that do not pass through $\rho$,
the first shell mirrors are the mirrors closest to the cusp $\rho$ in horocyclic distance.
The second shell mirrors are the next closest and so on.
The lemma below explicitly describes the roots in the first two shell.
One verifies easily that the first shell roots of $L$ are of the form
\begin{equation*}
i^r \bigl( \sigma; 1, p \bigl( \tfrac{1}{2} ( 1 - \tfrac{\sigma^2}{2} ) + \nu \bigr) \bigr)
\text{\; where\; }
\sigma \in \Lambda, \;
\nu \in \tfrac{i}{2}  \Z \text{\; and \;} \tfrac{2}{ i} \nu \equiv (1 - \tfrac{\sigma^2}{2} ) \bmod 2.
\end{equation*}
Writing $\nu =  i k - \tfrac{i}{2} ( 1 - \tfrac{\sigma^2}{2} ) $ we find that the first shell roots are
of the form
\begin{equation*}
 i^r (\sigma; 1, 1 - \tfrac{\sigma^2}{2} + i p  k), \text{\; where \;} \sigma \in \Lambda \text{\; and \;} k \in \Z.
\end{equation*}
One verifies easily that the second shell roots of $L$ are of the form
\begin{equation*}
i^r \bigl( \sigma; \bar{p},  \tfrac{1}{2}( 1 - \tfrac{\sigma^2}{2}) + \nu  \bigr)
\text{\;  where \;} \sigma \in \Lambda(2 \bmod 4) \text{\; and \;} \nu \in i  \Z.
\end{equation*}
%Writing $\nu = i k$, we find that the the height $\bar{p}$ roots are
%\begin{equation*}
%i^r (\sigma; \bar{p}, \tfrac{1}{2}(1 - \tfrac{\sigma^2}{2}) +  i k), 
%\text{\; where \;} \sigma \in \Lambda(2 \bmod 4) \text{\; and \;} k \in \Z.
%\end{equation*}
%
It is useful to note that if $s$ is a first or second shell root written as above and we change
$\nu$ to $\nu' \in \nu + i \Z$, then we again get a root in the same shell.
\end{topic}
Let $l = ( \lambda; h, *) \in L$. 
Let $s = (\sigma; m, *)$ be  a root of $L$. Assume $h, m \neq 0$.
The lemma below gives us condition for a reflection in $s$ to decrease the height
of $l$.
If $l$ is a root, then this is equivalent to saying that a reflection in $s$
brings the mirror $\B(l^{\bot})$ closer to $\rho$. 
The condition is conveniently expressed in terms of the quantity
\begin{equation*}
y = y(s, l ) = \abs{m}^2 \bigip{ \tfrac{s}{m} }{ \tfrac{l}{h} }.
\end{equation*}
\begin{lemma}
An order four reflection in $s$ decreases the height of $l$ if and only if
$y = y(s, l)$ belongs to $B( 1 + i, \sqrt{2}) \cup B( 1-i, \sqrt{2})$,
that is, the union of radius $\sqrt{2}$ open discs in the complex plane centered at $(1 \pm i)$.
\label{l-height-reduction}
\end{lemma}
\begin{proof}
Let $\xi = \pm i$ and let $R$ denote the $\xi$-reflection  in $s$. 
We calculate
\begin{align*}
\bigip{\rho}{R(\tfrac{l}{h} ) }
 = \bigip{R^{-1}(\rho) }{\tfrac{l}{h} } 
% = \Bigip{ \rho - \frac{(1 - \xi)}{2} \ip{s}{\rho} s}{\frac{l}{h} } 
 = \bigip{ \rho - \tfrac{1}{2} (1 - \bar{\xi}) \bar{m} \bar{p}  s}{\tfrac{l}{h} } 
 = p -  \tfrac{1}{2} (1 - \xi)  p y.
\end{align*}
The reflection $R$ brings $l^{\bot}$ closer to $\rho$ if and only if
$\bigabs{ \bigip{\rho}{\tfrac{l}{h}} } > 
\bigabs{ \bigip{\rho}{R(\tfrac{l}{h}) }} $, that is,
\begin{equation*}
\abs{p} >  \bigabs{p -  \tfrac{1}{2}(1 - \xi)   p y  }.
\end{equation*}
Multiplying both sides by $\abs{p^{-1}(1 - \bar{\xi})}$, the inequality becomes
\begin{equation*}
\sqrt{2}  >\bigabs{(1 - \bar{\xi})  -   y}
\end{equation*}
which is equivalent to $y$
lying in $B( 1 + i, \sqrt{2}) \cup B( 1-i, \sqrt{2})$.
\end{proof}
\begin{lemma}
Let $l = ( \lambda; h, p \bar{m}^{-1}  ( (l^2 -\lambda^2)/4 + \nu_l ))$ be a root of $L$.
\par
(a) If there exists $\sigma \in \Lambda$ with $(\sigma - \lambda/h)^2 \leq 2$ and $\abs{h} > 1$,
then a first shell reflection decreases the height of $l$.
\par
(b) If there exists $\sigma \in p^{-1}\Lambda(2 \bmod 4)$ with $(\sigma - \lambda/h)^2 \leq 1$
and $\abs{h} > \sqrt{2}$,
then a second shell reflection decreases the height of $l$.
\label{l-height-reduction-1-2}
 \end{lemma}
\begin{proof}
In part (a) (resp. (b)) we show that we can choose a root $s$ in the first (resp. second) shell 
such that $y = y(s, l)$ belongs to the rectangle $(0,2) \times [-i, i]$. 
The lemma then follows from \ref{l-height-reduction}, since this rectangle
is a subset of $B( 1 + i, \sqrt{2}) \cup B( 1-i, \sqrt{2})$.
\par
(a) Choose $\sigma \in \Lambda$
such that $\bigl(  \sigma - \lambda/h \bigr)^2 \leq 2$.
Consider a first shell root written in the form
$s =\bigl( \sigma; 1, p \bigl(  \tfrac{1}{2}(1 - \tfrac{1}{2}\sigma^2) + \nu \bigr) \bigr)$
as in \ref{l-roots-in-shells-1-2} with $\nu$ still to be determined.
Using \eqref{eq-complete-square}, we compute
\begin{equation}
\op{Re}(y)
=   1 + \tfrac{l^2}{ 2 \abs{h}^2} 
 - \tfrac{1}{2} \bigl(  \sigma - \tfrac{\lambda}{h} \bigr)^2 
\text{\; and \;} 
\op{Im}(y) = 
 \op{Im} \ip{\sigma}{ \tfrac{\lambda}{h}} +  \tfrac{ 2\nu_l}{\abs{h}^2} - 2\nu.
 \label{eq-blund-on-y-shell-1}
\end{equation}
Since $\abs{h} > 1$ and $l^2 = 2$, The choice of $\sigma$ ensures that $\op{Re}(y) \in (0,2)$.
From \ref{l-roots-in-shells-1-2},
note that we are free to choose $2 \nu/i$ either from $2 \Z$ or from $2 \Z + 1$.
So we can choose $\nu$ to ensure that
$\op{Im} (y)  \in [-i, i]$.
\par
(b) 
Choose $\sigma \in \Lambda(2 \bmod 4)$
such that $\bigl(  \sigma/\bar{p} - \lambda/h \bigr)^2 \leq 1$.
Consider a second shell root written in the form 
$s  = \bigl( \sigma; \bar{p}, \bigl( \tfrac{1}{2}(1 - \tfrac{1}{2}\sigma^2) + \nu \bigr) \bigr)$ 
as in \ref{l-roots-in-shells-1-2} with $\nu$ still to be determined.
Using \eqref{eq-complete-square}, we compute
\begin{equation}
\op{Re}(y)
 =  1 + \tfrac{ l^2 }{ \abs{h}^2}  - \bigl(  \tfrac{\sigma}{\bar{p} } - \tfrac{\lambda}{h} \bigr)^2
 \text{\; and \;} 
\op{Im}(y) =  
 2\op{Im} \ip{ \tfrac{\sigma}{\bar{p}}}{ \tfrac{\lambda}{h}} +  \tfrac{ 4 \nu_l}{\abs{h}^2} - 2\nu.
  \label{eq-blund-on-y-shell-2}
\end{equation}
Since $\abs{h} > \sqrt{2}$ and $l^2 = 2$, the choice of $\sigma$ ensures that $\op{Re}(y) \in (0,2)$. 
From \ref{l-roots-in-shells-1-2},
Note that we are free to choose $2 \nu/i$ from $2 \Z$. So we can choose $\nu$ to ensure that
$\op{Im} (y)  \in [-i, i]$.
\end{proof}
\begin{lemma}
Let $l = (\lambda; h, *)$ be a primitive null vector of $L$ with $h \neq 0$.
Assume that $l$ is not orthogonal to any root of $L$.
\par
(a) If there exists $\sigma \in \Lambda$ with $(\sigma - \lambda/h)^2 \leq 2$,
then a reflection in a first shell root decreases height of $l$.
\par
(b) If there exists $\sigma \in p^{-1}\Lambda(2 \bmod 4)$ with $(\sigma - \lambda/h)^2 \leq 1$,
then a reflection in a second shell root decreases the height of $l$.
\label{l-height-reduction-condition-at-cusp}
 \end{lemma}
\begin{proof}
The proof is almost identical to the proof of lemma \ref{l-height-reduction-1-2}. 
The computation is only slightly different. Equations \eqref{eq-blund-on-y-shell-1} and 
\eqref{eq-blund-on-y-shell-2} still hold but
now since $l^2 = 0$, we obtain $y(s, l) \in [0,1] \times [-i, i]$. This rectangle minus the origin
is a subset of $B( 1 + i, \sqrt{2}) \cup B( 1-i, \sqrt{2})$.
Since $l$ is assumed
to be not orthogonal to any root, we have $y(s, l) \neq 0$.
\end{proof}
\begin{definition}[The Heisenberg group of translations]
Let $\T$ be the group of automorphisms of $L$ that fix $\rho$ and act trivially
on $\rho^{\bot}/\rho$. One verifies that 
\begin{equation*}
\T 
= \lbrace T_{\lambda, z} \colon \lambda \in \Lambda, z \in i( \lambda^2/2 + 2 \Z) \rbrace
= \lbrace T_{\lambda, i(\lambda^2/2) + 2 i k} \colon \lambda \in \Lambda, k \in \Z \rbrace
\end{equation*}
where $T_{\lambda, z} \in \op{Aut}(L)$ is defined by
\begin{equation*}
T_{\lambda, z} (l; a, b) = (l + a \lambda; a,  -\bar{p}^{-1} \ip{\lambda}{l} + a \bar{p}^{-1} (z - \lambda^2/2) + b).  
\end{equation*}
Note that for each $\lambda \in \Lambda$, the integer $z/i$ either runs over $2 \Z$ or $2 \Z + 1$.
We call $\T$ the group of {\it translations}. 
One verifies that the translations form a discrete Heisenberg group 
whose multiplication is given by
\begin{equation}
T_{\lambda, z} T_{\lambda', z'} = T_{ \lambda + \lambda', z + z' + \op{Im} \ip{\lambda'}{\lambda} }.
\label{eq-multiplication-in-heisenberg-group}
\end{equation}
Note that the translations of the form $T_{0,z}$ are central,
$T_{\lambda, z}^{-1} = T_{-\lambda, -z}$,
and
\begin{equation}
T_{\lambda, z} T_{\lambda', z'} T_{\lambda, z}^{-1} T_{\lambda', z'}^{-1} 
= T_{0, 2 \op{Im} \ip{\lambda'}{\lambda}}.
\label{eq-commutator-of-translations}
\end{equation}
\end{definition}
\begin{lemma}
Let $R_1$ and $R_2$ be the $i$-reflections in the roots $r_1 = (0^n; 1, 1)$ and
$r_2 = (0^n; 1, i)$ respectively.
Let $\lambda \in \Lambda$. Choose $z$ such that $T_{\lambda, z} \in \T$.
Let $G$ be the group generated by the reflections in 
$T_{\lambda, z}( r_1)$, $T_{\lambda, z} (r_2)$, $r_1$, $r_2$.
Then $T_{\bar{p} \lambda, i \lambda^2} \in G$.
\label{l-r1-r2} 
\end{lemma}
\begin{proof}

Let $\beta$
be the automorphism of $L$ that is identity on $\Lambda$ and acts on $\cell$ as multiplication
by $-i$. Then one verifies that 
$R_1 R_2 =  \beta T_{0, -4i}$.
Since $T_{0, -4i}$ is a central translation, it follows that
\begin{equation*}
(R_1 R_2) T_{\lambda, z} (R_1 R_2)^{-1} 
= \beta 
T_{\lambda, z} \beta^{-1}
= T_{ i \lambda, z}.  
\end{equation*}
So $G$ contains
$T_{\lambda, z} (R_1 R_2) T_{\lambda, z}^{-1} (R_1 R_2)^{-1} = T_{\bar{p} \lambda, i \lambda^2}$.
\end{proof}
We finish this section by showing that $R(L)$ contains many translations, 
specifically, a finite index subgroup of $\T$.
\begin{corollary} 
(a) Fix a $\cG$-basis $\lambda_1, \dotsb, \lambda_n$ of $\Lambda$.
For each $j = 1, \dotsb, n$, fix $z_j \in \op{Im}(\C)$ such that $ T_{\lambda_j, z_j } \in \T$.
Let $G$ be the subgroup of $R(L)$ generated by reflections in the following set of roots:
\begin{equation*}
\lbrace r_k, T_{\lambda_j, z_j} (r_k), T_{i \lambda_j, z_j} (r_k) : k = 1, 2, j = 1, 2, \dotsb, n \rbrace.
\end{equation*}
Then $G$ contains a translation of the form $T_{\lambda, *}$
for each $\lambda \in p \Lambda$ and the central translations $T_{0; 4 i m }$ for all $m \in \Z$.
\par
(b) A full set of coset representatives for $(R(L) \cap \T) \backslash \T$ can be chosen from the finite set
$\mathbb{T}_* = \lbrace T_{\sigma, i \sigma^2/2}, T_{\sigma, 2 i + i (\sigma^2/2)} \colon
\sigma \in (\Lambda / p)^{\sim} \rbrace$ where $(\Lambda/p)^{\sim}$
is a full set of coset representatives for $\Lambda/ p$.
\label{l-R(L)-has-many-translations}
 \end{corollary}
\begin{proof}
Lemma \ref{l-r1-r2} implies that $G$ contains
a translation of the form $T_{  p \lambda_j, *}$ and a translation of the form
$T_{i p \lambda_j, *}$ for each $j$. 
From equation \eqref{eq-multiplication-in-heisenberg-group}, it follows
that $R(L)$ contains a translation of the form
$T_{\lambda, *}$  for all $\lambda \in p \Lambda$.
Since $\Lambda$ is $p$-modular, choose $\lambda, \lambda' \in p \Lambda$ 
with $\ip{\lambda'}{\lambda} = 2 p$. Then equation \eqref{eq-commutator-of-translations}
implies that  $G$ also contains
$T_{0, 4 i}$. So $G$ contains the central translations of the form $T_{0 , 4 i m}$ for all $m \in \Z$.
\par
(b) Let $T = T_{\lambda', *} \in \T$. Choose $\sigma \in (\Lambda/ p)^{\sim}$
such that $\sigma - \lambda' = p \lambda$ for some $\lambda \in \Lambda$.
By part (a) we can choose a translation $T_1$ in $R(L)$ having the form
$T_1 = T_{p \lambda, *}$. Then $T_1 T = T_{p \lambda, *} T_{\lambda', *} 
= T_{p \lambda + \lambda' , *} = T_{\sigma, *}$. So $T_1 T = T_{\sigma, 2 i k + i (\sigma^2/2)}$
for some $k \in \Z$. Choose $m \in \Z$ such that $k + 2m \in \lbrace 0, 1 \rbrace$.
Let $T' = T_{0, 4 m i} T_1$. Then $T' \in R(L)$ and 
 $T' T = T_{\sigma,2i( k + 2m ) +  i (\sigma^2/2) } \in \mathbb{T}_*$.
\end{proof}
\section{Reflection group of the $p$-modular $\cG$-lattice of signature $(9,1)$}

\begin{lemma}
There is a unique $p$-modular Gaussian lattice of signature $(4n+1,1)$.
\label{l-pmodular}
\end{lemma}
Lemma \ref{l-pmodular} quickly follows from the uniqueness of even self-dual $\Z$-lattice of
signature $(8n+2,2)$;
as in the proof of lemma 2.6 of \cite{TB:el}.
The only implication of lemma \ref{l-pmodular} that we need is the isomorphism 
$4 D_4^{\cG} \oplus \cell \simeq \bwg \oplus \cell$. However, for our computational needs,
we shall exhibit an explicit isomorphism $4 D_4^{\cG} \oplus \cell \simeq \bwg \oplus \cell$
in \ref{t-computer}.
So we omit the proof of lemma \ref{l-pmodular}.
\begin{topic}{\bf Notation: }
In the last section, $\Lambda$ denoted any positive definite $p$-modular Gaussian lattice.
From here on, unless otherwise stated, we
specialize to the case 
\begin{equation*}
\Lambda = \bwg
\text{\; and \;} 
L  = \cG_{9,1}  \simeq 4 D_4^{\cG} \oplus \cell \simeq \Lambda \oplus \cell.
\end{equation*}
Both descriptions of $L$ are going to be useful for us. 
In this section, unless otherwise stated, we identify $L = \Lambda \oplus \cell$.
In the next section we shall use the identification $L = 4 D_4^{\cG} \oplus \cell$. 
Let $\rho = (0; 0,1) \in L$. Since $\Lambda$ has no root,
there are no mirrors through the cusp $\rho$. As before, the mirrors closest to $\rho$ are called the 
first shell mirrors; the mirrors that are next closest to $\rho$ are called second shell mirrors, and so on.
The corresponding roots are called first shell roots, second shell roots etc. 
\end{topic}
Our first goal is to show that the projective reflection group $\Gamma = PR(L)$ is an arithmetic lattice
in $P U(9,1)$.
The plan of the proof follows
 Theorem 4.1 of \cite{A:Leech}.
\begin{lemma}
The reflection group $R(L)$ acts transitively on
 primitive null vector of $L$ (considered up to $4$-th roots of unity)
 that are not orthogonal to any roots.
 \label{l-R(L)-is-transitive-on-BW-cusps}
\end{lemma}
\begin{proof}
Let $z = (\zeta; h, *)$ be a primitive null vector of $L$ that is not orthogonal to any root.
Suppose $h \neq 0$.
Lemma \ref{l-covering-lemma} implies that either there exists $\sigma \in \Lambda$ such that
$(\sigma - \zeta /h)^2 \leq \sqrt{2}$ or there exists $\sigma \in p^{-1} \Lambda (2 \bmod 4)$ such that
$(\sigma - \zeta/h)^2 \leq 1$. 
Lemma \ref{l-height-reduction-condition-at-cusp} then shows that the height of $z$ can be reduced by a
reflection in some root
(lemma \ref{l-height-reduction-condition-at-cusp} only uses reflections in the first and second shell roots
but this is  irrelevant for this argument).
By induction, it follows that a finite sequence of reflection can bring $z$ to a null vector
of the form $z_1= (\zeta_1; 0, *)$. Now, $z_1^2 = 0$ implies $\zeta_1^2 = 0$.
Since $z_1$ is primitive, it follows that $z_1$ is an unit multiple of $\rho$.
\end{proof}
\begin{theorem}
 $\abs{\Gamma \backslash P \! \op{Aut}(L)} \leq 2^{9} \abs{ \op{Aut}(\Lambda)}$. In particular $\Gamma$
 is arithmetic.
 \label{th-Gamma-is-arithmetic}
\end{theorem}
\begin{proof}
Let $\rho_1 = (0; 1, 0)$.
Take $g \in \op{Aut}(L)$. Then $g \rho$ is a primitive null vector that is not orthogonal
to any root.
Lemma \ref{l-R(L)-is-transitive-on-BW-cusps} 
implies that there exists $g_1 \in R(L)$ such that $i^{-m} g_1^{-1} g \rho = \rho$
for some $m \in \Z/4$.
So $i^{-m} g_1^{-1} g \rho_1$ is a null vector of the form $(* ;1, *)$. One verifies that
the group of translations act
transitively on the null vectors of the form $(*; 1, *)$. So there exists a translation
$T \in \op{Aut}(L)$ such that $T^{-1} i^{-m} g_1^{-1} g$ fixes $\rho$ and $ \rho_1$.
So $T^{-1}  i^{-m} g_1^{-1} g = \alpha \in \op{Aut}(\Lambda)$, a finite group.
So $g  = g_1 T \alpha i^m$. By lemma \ref{l-R(L)-has-many-translations}(b), there exists a
 $g_2 \in R(L)$ and $t \in \mathbb{T}_*$ such that 
  such that $T = g_2 t $.
   So $g = g_1 T \alpha i^{m}= g_1 g_2 t \alpha i^m \in R(L) t \alpha i^m$.
 It follows that a full set of coset representatives for $\Gamma \backslash P \! \op{Aut}(L)$ can be
 chosen from the finite set 
$ \lbrace  t \alpha  \colon t \in \mathbb{T}_*, \alpha \in\op{Aut}(\Lambda) \rbrace$.
\end{proof}
The goal of the rest of the section is to find a finite set of generators for $R(L)$.
\begin{lemma}
The $i$-reflections in the first and second shell roots generate $R(L)$.
\label{l-first-two-shells-generate}
\end{lemma}
\begin{proof}
The argument is similar to the proof of lemma \ref{l-R(L)-is-transitive-on-BW-cusps} except that one needs
lemma \ref{l-height-reduction-1-2} instead of lemma \ref{l-height-reduction-condition-at-cusp}. 
From lemma \ref{l-covering-lemma}, we know that
closed balls of radius $\sqrt{2}$ around vectors in $\Lambda$ together with closed balls or
radius $1$ around the vectors in $p^{-1} \Lambda(2 \bmod 4)$
cover the underlying real vector space of $\Lambda$. So the lemma follows from
\ref{l-height-reduction-1-2}, parts (a) and (b) using induction on height of a root.
\end{proof}
Lemma \ref{l-first-two-shells-generate} give us an infinite set of reflections that generate $R(L)$. The next lemma shows that
an explicit finite subset of these reflections are enough to generate $R(L)$.
To list the roots of these generating reflections, fix a $\cG$-basis $\lambda_1, \dotsb, \lambda_8$ of $\Lambda$.
For each $j = 1, \dotsb, 8$, fix a $z_j \in \op{Im}(\C)$ such that $ T_{\lambda_j, z_j } \in \T$.
Recall that if $m \in \cG$, then $(\Lambda/ m)^{\sim}$ denotes
a full set of coset representatives for $\Lambda/ m$. Define
\begin{itemize}
\item $S_0 = \lbrace r_k, T_{\lambda_j, z_j} (r_k), T_{ i \lambda_j, z_j}(r_k) \colon j = 1, \dotsb, 8, k = 1, 2 \rbrace$, 
\item $S_1 = \lbrace  \bigl( \sigma; 1, 1 - \tfrac{\sigma^2}{2} + i p k  \bigr) \colon 
\sigma \in (\Lambda/ p)^{\sim},  k = 0, 1 \rbrace$.
\item $ S_2 = \lbrace \bigl(\sigma; \bar{p}, \tfrac{1}{2} (1 - \tfrac{\sigma^2}{2}) + i k \bigr) \colon
\sigma \in (\Lambda/2)^{\sim} \cap \Lambda(2 \bmod 4), k = 0,1 , 2, 3 \rbrace$.
\end{itemize}
From the discussion in \ref{t-q}, recall that all the vectors in a coset in $\Lambda/2$
either have norm $0 \bmod 4$ or have norm $2 \bmod 4$. So
$(\Lambda/2)^{\sim} \cap \Lambda(2 \bmod 4)$ is a set of representatives
of the cosets that consist of vectors of norm $2 \bmod 4$.
\begin{lemma}
The $i$-reflections in the roots in $S_0 \cup S_1 \cup S_2$ generate $R(L)$. 
\label{l-finite-list-of-generators-for-R(L)}
\end{lemma}
\begin{proof}
Let $G$ be the group generated by the reflections listed. 
Since $G$ contains the reflections in the roots in $S_0$,
lemma \ref{l-R(L)-has-many-translations}(a)
implies that $G$ contains
a translation of the form $T_{p \lambda, *}$ for all $\lambda \in  \Lambda$
and the central translations of the form $T_{0 , 4 i n}$ for $n \in \Z$.
We make the following claim: 
\par
{\it Claim: 
If $s$ is a root of the form $(*; 1, *)$ (resp. $(*; \bar{p}, *)$), then 
there exists a translation $T \in G$ such that $T s$ belongs to $S_1$ (resp. $S_2$).
}
\par
Let $s$ be a root of the form $s = (\sigma; 1, *)$.
Choose $\sigma_0 \in (\Lambda/ p)^{\sim}$ such that
$\sigma_0 - \sigma = p \lambda$ for some $\lambda \in \Lambda$.
Choose a translation of the form $T_{p \lambda, *} \in G$ that takes 
$s$ to a root of the form $s' = (\sigma_0; 1, *)$. Next, one can choose
a central translation in $G$ that takes  $s'$
to a root in $S_1$. This proves the claim for roots
of the form $(*; 1, *)$ \footnote{$\T$ acts simply transitively
on the roots of the form $(*;1,*)$.
So the argument here is
essentially a repeat of the proof of
lemma \ref{l-R(L)-has-many-translations}(b).}.
The argument for roots of the form
$(*; \bar{p}, *)$ is similar. 
\par
 Now let $R$ be an $i$-reflection in a root in the first or second shell.
Then there exists a root $s$ of the form $ (*; 1, *)$ or $ (*; \bar{p}, *)$
such that $R = R_s$. Choose  $T \in G$ such that $T s \in S_1 \cup S_2$.
So $R_{T s} \in G$. It follows that $R_s = T^{-1} R_{T s} T \in G$.
Thus, $G$ contains the $i$-reflections in all the roots in the first two shells.
 Lemma \ref{l-first-two-shells-generate} completes the proof. 
\end{proof}
%
%
%*******************************************************************************************************
%
%*******************************************************************************************************
%
\section{The thirty two mirrors closest to a point in $\C H^9$}
The goal of this section is to find nice generators and relations for $R(L)$,
analogous to the Coxeter generators and relations of Weyl groups.
This is only a rough analogy; for example the generators are 
not a minimal set and we do not know if the our relations are sufficient to give a presentation
of $R(L)$. However, the analogy is reinforced because similar phenomenon repeats
for other complex hyperbolic reflection groups of interesting lattices, as illustrated by Theorem \ref{th-1}.
In particular, everything in this section closely parallels the results of \cite{TB:el} and \cite{TB:ql}.
\par
As for Coxeter groups, our generators and relations can be encoded in a diagram $D$, a sort 
of Coxeter-Dynkin diagram of $R(L)$. This diagram can be defined from the intersection pattern of
a configuration of points and hyperplanes in the finite vector space $\F_2^4$.
We shall define the lattice $L$ starting from $D$, rather
like defining the root lattice from a Cartan matrix.
We start by describing these points and hyperplanes and by working out 
the symmetries of this configuration.
\begin{topic}{\bf A configuration of points and hyperplanes in $\F_2^4$: }
 Let 
  \begin{equation*}
 g_1 = (1,0,0,0), \dotsb, g_4 = (0,0,0,1)
 \end{equation*}
be the standard basis vectors of $\F_2^4$. 
The sixteen points of $\F_2^4$ are named $ a$, $c_j$, $g_j$, $e_{i j}$, $z$
where
\begin{equation*}
a = (1,1,1,1), \;\; c_i = a + g_i, \;\; e_{ i j} = a + g_i + g_j, \;\; z = (0,0,0,0)
\end{equation*}
and $1 \leq i < j \leq 4$.
Next, we name sixteen hyperplanes in $\F_2^4$.
Let 
\begin{equation*}
d_k = \lbrace (x_1, x_2, x_3, x_4) \in \F_2^4 \colon x_k = 0 \rbrace
\text{\; and \;} 
f_k = \lbrace (x_1, x_2, x_3, x_4) \in \F_2^4 \colon x_k = \sum_j x_j \rbrace.
\end{equation*}
So $\cK_0 = \lbrace d_1, \dotsb, d_4, f_1, \dotsb, f_4 \rbrace$
is the set of homogeneous hyperplanes in $\F_2^4$ not containing $a$.
Let 
\begin{equation*}
b_k = \F_2^4 - f_k \text{\; and  \;} h_k = \F_2^4 - d_k.
\end{equation*}
Finally let 
$\cK = \lbrace d_k, f_k, b_k, h_k \colon k = 1, 2, 3, 4 \rbrace$
be the set of translates of $\cK_0$ and
let 
\begin{equation*}
D = \F_2^4 \cup \cK.
\end{equation*}
Let $Q_+$ be the subgroup of the group of affine transformations
of $\F_2^4$ that preserve $D$. One verifies that $Q_+ \simeq 2^4 : (2^3 : L_3(2))$
where the $2^4$ comes from the translation action of $\F_2^4$ on itself
and the $2^3 : L_3(2)$ is the stabilizer of $a$ in $L_4(2)$.
It will be useful to note that the symmetric group $S_4$ acting by coordinate permutation
on $\F_2^4$ fixes $a$. So this $S_4$ is a subgroup of $Q_+$.
\par
The diagram for our reflection group $R(L)$ is  shown in figure \ref{fig-X3333}.
It is a graph with vertex set $D$
and with two kinds of edges as shown in figure \ref{fig-X3333}.
Two vertices $u$ and $v$ are joined by a dotted edge if $v=t_a(u)$
where $t_a$ is the automorphism of $D$ induced by the translation $t_a: \F_2^4 \to \F_2^4$
defined by $t_a(v) = a + v$.
If we ignore the dotted edges, then $D$ should be thought of a directed  bipartite graph
with a directed edge from $v$ to $u$ if $v \in \cK$, $u \in \F_2^4$ and $u$ is incident
on $v$. The automorphism group $Q_+$ acts on this graph $D$ preserving both kinds of 
edges.
\end{topic}
\begin{lemma}
Let $L^{\circ}$ be the Gaussian lattice of rank $32$ with a basis 
$\lbrace s_v^{\circ} \colon v \in D \rbrace$ indexed by $D$ and inner product
defined by
\begin{equation}
\ip{s_u^{\circ}}{s_v^{\circ}} = 
\begin{cases}
2 & \text{\; if \;} u = v, \\
p & \text{\; if \;} u \in \F_2^4, v \in \cK \text{\; and \;} u \in v, \\
\bar{p} & \text{\; if \;} u \in \cK, v \in \F_2^4 \text{\; and \;} v \in u, \\
-2 & \text{\; if \;} v = t_a(u), \\
0 & \text{\; otherwise}.
\end{cases}
\label{eq-ipD}
\end{equation}
Then $L \simeq L^{\circ}/ \op{Rad}(L^{\circ})$.
\label{l-L-from-D}
\end{lemma}
\begin{proof}
Identify $L = 4 D_4^{\cG} \oplus \cell$. We claim that there are
 $32$ roots $\lbrace s_v \colon v \in D \rbrace$ in $L$
such that the inner products between them are
governed by $D$ as in equation \eqref{eq-ipD}. In other words,
there is an inner product preserving linear map $L^{\circ} \to L$ by
that sends $s_v^{\circ}$ to $s_v$. Recall that there is an obvious $S_4$ action
on $D$. The symmetric group $S_4$ also acts on $4D_4^{\cG} \oplus \cell$
by permuting the four copies of $D_4^{\cG}$. The map 
$s_v^{\circ} \mapsto s_v$ is going to be $S_4$ equivariant.
Define
\begin{align*}
 s_a          &=\ph [ \ph 0, 0,   \ph 0,  0,  \ph 0,  0, \ph 0,  0 ; \ph\ph -1, -1],\\
s_{c_1}      &=\ph [      -1, 1,   \ph 0,  0,  \ph 0,  0, \ph  0,  0 ; \ph\ph\ph \;\; 0,  \ph0],\\
s_{e_{12}} &=     -[  \ph 1, 1,  \ph 1,  1,  \ph 0,  0, \ph 0, 0 ; \ph\ph\ph \;\; i, \ph1],\\
s_{  g_1 }   &=     -[  \ph 0, 2,  \ph 1,  1,   \ph 1,  1, \ph 1,  1; \ph\ph\ph 2 i, \ph2],\\
 s_{ z    }    &=     -[  \ph 1, 1,  \ph 1, 1,   \ph 1, 1,   \ph 1, 1; -1 + 2 i,\ph 1],\\
s_{f_1  }     &= \ph[  \ph 0, 0,  \ph 0,  p,  \ph 0,  p,  \ph  0,  p; \ph\ph \ph i p, \ph p], \\
 s_{b_1  }   &= \ph[  \ph 0, p,  \ph 0,  0,   \ph 0,  0,  \ph 0,  0; \ph\ph -1, \ph 0],\\
 s_{d_1 }   &=  \ph[       -p, 0,  \ph 0,  0,   \ph 0,  0, \ph  0,  0;\ph\ph \ph \;\;  0, \ph 0],\\
s_{ h_1}    &= \ph[   \ph p, p,  \ph 0,  p,   \ph 0,  p, \ph  0,  p; \ph \; p-3,\ph p].
\end{align*}
The other roots are obtained by using the $S_4$ symmetry (that is, by permuting the
four copies of $D_4^{\cG}$).
For example 
$s_{c_2}   =[  0, 0, \;  -1, 1, \; 0,  0,\;   0,  0 ; \;\; 0,  0]$.
The lemma is proved once one verifies that these roots have the required inner products.
\end{proof}
\begin{remark}
We want to mention how we came across the $32$ roots $\lbrace s_v \colon v \in D \rbrace$.
The complex reflection group of $D_4^{\cG}$ is generated by two braiding $i$-reflections.
In other words, $R(D_4^{\cG})$ has Dynkin diagram $A_2$ with each vertex having order $4$.
Write $L_* = 3D_4^{\cG} \oplus \cell$.
Affinizing and hyperbolizing (in the sense of \cite{CS:SPLAG}, chapter 30) one gets a diagram $Y_{3 3 3}$ inside 
$R(L_*)$. The graph $Y_{3 3 3}$ is a maximal sub-tree in $\op{Inc}(P^2(\F_2))$, which means the incidence 
graph of finite projective plane $P^2(\F_2)$. 
One can extend the 
$Y_{3 3 3}$ diagram uniquely to a $ \op{Inc}( P^2(\F_2))$ diagram in $R(L_*)$.
The diagram automorphisms $ L_3(2) : 2$ act on $\mathbb{B}(L_*)$ with a unique fixed point
$\tau_*$ and the $14$ mirrors corresponding to the vertices of $ \op{Inc}( P^2(\F_2))$
are exactly the mirrors closest to $\tau$.
We tried to prove the results similar to theorem \ref{th-1} for the reflection group $R(L_*)$ but
 could not make the arguments work because
the covering radius of the lattice $3D_4^{\cG}$ is not small enough.
However, these arguments  work for 
$L = 4D_4^{\cG} \oplus \cell$ because of the alternative description $L \simeq \bwg \oplus \cell$.
In the root system of 
$L = 4D_4^{\cG} \oplus \cell$, we can naturally extend the $ \op{Inc}( P^2(\F_2))$
diagram, first by using 
the obvious $S_4$ symmetry permuting the four copies of $D_4^{\cG}$ and then by looking for a 
regular graph. This leads to a set of $32$ root diagram $D$ in $L$.
\end{remark}
\begin{topic}{\bf Linear relations among the $32$ roots:} 
One could also prove lemma \ref{l-L-from-D} by working out the $22$ dimensional radical of $L^{\circ}$.
It is useful for us to at least write down
enough linear relations among the $32$ roots $\lbrace s_v \colon v \in D \rbrace$,
where enough means that the corresponding vectors of $\op{Rad}(L^{\circ})$ span
 $\op{Rad}(L^{\circ}) \otimes \C$. 
If both $v, w \in \F_2^4$ or both $v,  w \in \cK$, then we have the relation
\begin{equation}
s_v + s_{t_a(v)}  = s_w + s_{t_a(w)}.
\label{eq-translation-relation}
\end{equation}
We define
\begin{equation*}
p_{\infty} = s_v + s_{t_a(v)} \text{\; if \;} v \in \F_2^4,
\end{equation*}
and
\begin{equation*}
l_{\infty} = s_v + s_{t_a(v)} \text{\; if \;} v \in \cK.
\end{equation*}
Using  \eqref{eq-ipD}, one verifies that $p_{\infty}$ and $l_{\infty}$ are primitive
null vectors of $L$.
Explicitly, one computes
\begin{align*}
p_{\infty} = -(1, 1, 1, 1, 1, 1, 1, 1;  2i , 2)
\text{\; and  \;}
l_{\infty} = (0,p,0,p,0,p,0,p ; p-3, p).
\end{align*}
Further, for each $u \in \F_2^4$ and for each $w \in \cK$, we have the relations 
\begin{equation}
- 2 (1 + i) s_u + \sum_{v \in \cK : u \in v} s_v = 4 l_{\infty} - (1 + i) p_{\infty},
\label{eq-points-from-hyperplanes}
\end{equation}
and
\begin{equation}
- 2 (1 - i ) s_w + \sum_{v \in \F_2^4: v \in w} s_v = 4 p_{\infty} - (1 - i) l_{\infty}.
\label{eq-hyperplanes-from-points}
\end{equation}
To verify the relation \eqref{eq-translation-relation}, one just checks, using
\eqref{eq-ipD}, that both sides of \eqref{eq-translation-relation} have the same
inner product with each vector in $\lbrace s_v \colon v \in D \rbrace$. 
The relations in \eqref{eq-points-from-hyperplanes} 
and \eqref{eq-hyperplanes-from-points} can be verified similarly, or by direct computation.
\end{topic}
\begin{topic}{\bf The configuration of the $32$ mirrors in complex hyperbolic space:}
\label{t-configuration}
One verifies that the mirrors $\lbrace s_v^{\bot} \colon v \in \F_2^4 \rbrace$ meet at the
cusp determined by $p_{\infty}$ and the mirrors 
$\lbrace s_v^{\bot} \colon v \in \cK \rbrace$ meet at the cusp determined by $l_{\infty}$. 
The group $Q_+ = 2^4 \colon (2^3 \colon L_3(2))$ transitively permutes these two sets
of sixteen mirrors and fixes the $\C H^1$ spanned by $p_{\infty}$ and $l_{\infty}$.
The set of thirty two mirrors $\lbrace s_v^{\bot} \colon v \in D \rbrace$ has an 
extra symmetry, that is described below.
Let $\sigma_{\cK} : \cK \to \F_2^4$ be the bijection
\begin{equation*}
\sigma_{\cK}(\lbrace (x_1, \dotsb, x_4) \in \F_2^4: \sum_{j} u_j x_j = \epsilon \rbrace) = (u_1, u_2, u_3, 1) + (u_4 + \epsilon) (1,1,1,1).
\end{equation*}
Define $\sigma_D: D \to D$ by $\sigma_D \vert_{\cK} = \sigma_{\cK}$ and $\sigma_D \vert_{\F_2^4} = \sigma_{\cK}^{-1}$.
One verifies that the involution $\sigma_D$ preserves the incidence relations between points and
hyperplanes and
commutes with the action of the translation $t_a$
\footnote{If the definition of $\sigma_{\cK}$ seem too ad-hoc,
the following discussion may help. Identify
$\F_2^4$ and $\cK$  with two affine hyperplanes
$\mathfrak{F} = \lbrace x \in \F_2^5 \colon x_5 = 1 \rbrace$ and 
$\mathfrak{K} = \lbrace x \in \F_2^5 \colon x_1 + x_2 + x_3 + x_4 = 1  \rbrace$
in $\F_2^5$ respectively
via  
\begin{equation*}
i_{\mathfrak{F}}: v \mapsto \bigl( \begin{smallmatrix} v \\ 1 \end{smallmatrix} \bigr) \text{\; and \;} 
i_{\mathfrak{K}}:  \lbrace x \in \F_2^4 \colon  u^T  x = \epsilon \rbrace \mapsto \bigl( \begin{smallmatrix} u \\ \epsilon \end{smallmatrix} \bigr).
\end{equation*}
Note that $v \in \F_2^4$ belongs to the hyperplane
$\lbrace x \in \F_2^4 \colon u \cdot x = \epsilon \rbrace$
if and only if 
$\bigl( \begin{smallmatrix} v \\ 1 \end{smallmatrix} \bigr)$
and
$\bigl( \begin{smallmatrix} u  \\ \epsilon \end{smallmatrix} \bigr)$
are orthogonal with respect to the standard inner product
of $\F_2^5$. We need to choose an appropriate $\sigma_{\mathfrak{K}} \in L_5(2)$ 
taking $\mathfrak{K}$ to $\mathfrak{F}$.
For this, let $J_{m,n}$ denote the $m \times n$ matrix all whose entries are equal to $1$.
Let $H = \bigl( \begin{smallmatrix} 0 & 1 \\ 1 & 0 \end{smallmatrix} \bigr)$. 
Let $I_n$ denote the $n \times n$ identity matrix. Define 
$\sigma_{\mathfrak{K}} = \bigl( \begin{smallmatrix} I_3 & J_{3,2} \\ J_{2,3} & H \end{smallmatrix} \bigr) \in L_5(2)$.
Verify that 
$\sigma_{\mathfrak{K}}^{-1}: \mathfrak{F} \to \mathfrak{K}$
 and $\sigma_{\mathfrak{K}} : \mathfrak{K} \to \mathfrak{F}$ are mutually inverse bijections
 and $\sigma_{\mathfrak{K}}^2$ is an involution.
Let $f \in \mathfrak{F}$ and $k \in \mathfrak{K}$. 
Since $\sigma_{\mathfrak{K}}$ is self adjoint,
$ \sigma_{\mathfrak{K}}^{-1} f$ is orthogonal to $\sigma_{\mathfrak{K}} k$ if and only $k$ is orthogonal to $f$.
Let 
$\sigma_{\cK} =i_{\mathfrak{F}}^{-1} \circ \sigma_{\mathfrak{K}} \circ i_{\mathfrak{K}}$.
Then $\sigma_{\cK} : \cK \to \F_2^4$ is a bijection such that $v \in h$ if and only if $\sigma(h) \in \sigma(v)$
for all $v \in \F_2^4$ and for all $h \in\cK$. In other words, $\sigma_{\cK}$ and $\sigma_{\cK}^{-1}$
interchanges the points in $\F_2^4$ with the hyperplanes in $\cK$ preserving the incidence relations.
For $w \in \F_2^4$, the translation $t_w$ acts on $\fF$ and $\fK$ as the matrices 
$t_w \vert_{\fF} = \bigl( \begin{smallmatrix} I_4 & w \\ 0 & 1 \end{smallmatrix} \bigr)$
and
$t_w \vert_{\fK} = \bigl( \begin{smallmatrix} I_4 & 0 \\ w^T & 1 \end{smallmatrix} \bigr)$.
One verifies that 
$\bigl( t_w \vert_{\fF} \bigr) \sigma_{\fK}  
= \sigma_{\fK} \bigl( t_w \vert_{\fK} \bigr)$ if
and only if $w = a$ or $w = 0$.
}. 
It follows that $\sigma_D$ is an orientation reversing involution of $D$.
In other words, the involution $\sigma_D$ acts on $D$ by preserving both kinds of edges and
reversing the orientation on the solid edges.
Define 
\begin{equation*}
\sigma_{L^{\circ}} : L^{\circ} \to L^{\circ}
\text{\; by \;}
\sigma_{L^{\circ}} (s_v^{\circ}) =
\begin{cases}
 s_{ \sigma_D(v)}^{\circ} & \text{\; if  \;} v \in \cK, \\ 
 -i s_{ \sigma_D(v)}^{\circ}. & \text{\; if \;}   v \in \F_2^4.
 \end{cases}
\end{equation*}
From definition of inner product on $L^{\circ}$, it follows
that $\sigma_{L^{\circ}}$ is an inner product preserving isomorphism of $L^{\circ}$
whose square is multiplication by $-i$. So $\sigma_{L^{\circ}}$ 
descends to define an automorphism $\sigma = \sigma_L$ 
of $L$ and an involution of $\B(L)$, also denoted by $\sigma$.
 This involution $\sigma$ interchanges the sixteen mirrors meeting at $p_{\infty}$
with the sixteen mirrors meeting at $l_{\infty}$. 
One verifies that the group $Q$ generated by $\sigma$ and $Q_+$
permutes the $32$ mirrors transitively and fixes a unique point in
$\B(L)$ which can be represented by the vector
\begin{equation*}
\tau = e^{-\pi i/4} l_{\infty} - p_{\infty}.
\end{equation*}
The $32$ mirrors are all equidistant from $\tau$. Let $d_0$ be this distance. We compute
\begin{equation*}
d_0 = d(\tau, s_v^{\bot}) \approx 0.4090 \text{\; for all \; } v \in D.
\end{equation*}
As already mentioned, we are using the same notation for a vector, say $\tau$ (resp.
a hyperplane, say $s_j^{\bot}$)
in $L$ and the point  (resp. hyperplane) in $\B(L)$ it determines. 
Let $B$ be a small horoball around $p_{\infty}$ not containing $\tau$ and let
$B'$ be the image of $B$ under any automorphism of $L$ taking $p_{\infty}$ 
to $l_{\infty}$. Then $\tau$ is the point on the real geodesic joining $p_{\infty}$
and $l_{\infty}$ that is equidistant from $B$ and $B'$.
We should think of $\tau$ as the mid-point between $p_{\infty}$ and $l_{\infty}$.
\end{topic}
\begin{theorem} The $32$ mirrors $\lbrace s_v^{\bot} \colon v \in D \rbrace$
 are precisely the mirrors closest to $\tau$.
In particular, $\tau$ does not lie on any mirror.
\label{th-simple-mirrors}
\end{theorem}
Let $r$ be any root of $L$ such that $d(\tau, r^{\bot} ) \leq d_0$.
The lemma below gives us conditions that allow us to restrict the possibilities
for $r$ to a finite set.
\begin{lemma}
(a) Let $w$ be a root of $L$. Then 
\begin{equation*}
\abs{ p^{-1} \ip{w}{r} }^2 \leq 2 \cosh^2( d_0 + d( \tau, w^{\bot})).
\end{equation*}
\par
(b) Let $w$ be a primitive null vector of $L$. 
From \ref{def-horocyclic-distance}, recall that $d_w(\tau)   =\tfrac{1}{2} \log (\abs{\ip{w}{\tau}}^2/(-\tau^2))$.
One has
\begin{equation*}
\abs{ p^{-1} \ip{w}{r} }^2 \leq e^{2 (d_0 +  d_{w}(\tau))}.
\end{equation*}
In each case, note that $p^{-1}\ip{w}{r} \in \cG$.
So $\abs{p^{-1} \ip{w}{r}}^2$ belongs to $\lbrace 0, 1, 2, 4, 5, 8, 9, \dotsb \rbrace$.
Thus, in each case, there are finitely many possibilities for $\ip{w}{r}$.
\label{l-bounds}
\end{lemma}
\begin{proof}
(a) If the hyperplanes $w^{\bot}$ and $r^{\bot}$ meet in 
$\B(L) \cup \partial \B(L)$, then $\abs{ p^{-1} \ip{w}{r} }^2 \leq 2$, so the required inequality holds trivially.
Otherwise,  the triangle inequality implies
\begin{equation*}
d(r^{\bot}, w^{\bot}) \leq d( r^{\bot}, \tau)  + d(\tau, w^{\bot})  \leq d_0 + d(\tau, w^{\bot}).
\end{equation*}
Part (a) follows using the formula given in \ref{t-complex-hyperbolic-space} for
distance between two hyperplanes.
\par
(b) Let $p_{\tau}$  be the projection of $\tau$  on $r^{\bot}$.
Then $d(\tau, p_{\tau}) = d(\tau, r^{\bot}) \leq d_0$.
Choose a small horoball $B$
around $w$ that does not meet $\tau$ and $r^{\bot}$.
Let $p_w$ be the point of $r^{\bot}$ that is nearest $B$. In other words $p_w$ is the
projection of $w$ on $r^{\bot}$. Then
 $p_w$ is closer to $w$ than $p_{\tau}$, that is, $d_w(p_w) \leq d_w(p_{\tau})$.
From the ideal triangle inequality, we have
\begin{equation*}
\tfrac{1}{2} \log \bigl( \tfrac{\abs{ \ip{w}{p_w} }^2 }{ -p_w^2 } \bigr) 
= d_w(p_w) \leq d_w(p_{\tau}) \leq d_w(\tau) + d(\tau, p_{\tau})  \leq d_w(\tau) + d_0.
\end{equation*}
So 
\begin{equation*}
\abs{ \ip{w}{p_w} }^2/ (-p_w^2) \leq e^{ 2(d_0 + d_w(\tau))}. 
\end{equation*}
The projection $p_w$ of $w$ on $r^{\bot}$ 
is represented by the intersection of $\C w + \C r$ and $r^{\bot}$. So
$p_w$ can be represented by the vector $p_w = w - \ip{r}{w} r/ r^2$.
One computes that $p_w^2   = - \abs{\ip{r}{w}}^2/ r^2$
and
$\ip{w}{p_w}^2/ (-p_w^2)  = \abs{\ip{r}{w}}^2/ r^2$.
Part (c) follows.
\end{proof}
Lemma \ref{l-bounds}(a) implies, in particular, that
$\abs{ p^{-1} \ip{r}{s_v} }^2 \leq 2 \cosh^2( 2 d_0) \approx 3.6642$ for all $v \in D$.
So 
\begin{equation*}
p^{-1} \ip{r}{s_v} \in \cG(\leq 2) \text{\; for all \;} v \in D.
\end{equation*}
(Recall that $\cG(\leq k)$ denotes the set of elements of $\cG$ of norm $\leq k$). 
To obtain further restrictions on $r$,
it will be convenient to use the basis $v_1, \dotsb, v_{10}$ for $L \otimes \C$ given below:
\begin{align*}
(v_1, v_2, \dotsb, v_{10}) = ( -s_{d_1}, s_{b_1},-s_{d_2}, s_{b_2}, -s_{d_3}, s_{b_3}, -s_{d_4}, s_{b_4}, (0^8; 1, 0), l_{\infty}).
\end{align*}
The roots $s_{d_i}, s_{b_i}$ were defined in the proof of lemma \ref{l-L-from-D}.
The inner products between $v_1, \dotsb, v_{10}$ are described as follows:
$v_1, \dotsb, v_8$ are eight roots that form an orthogonal basis for a 
maximal positive definite subspace of $L \otimes \C$ whose orthogonal complement
 has a basis consisting of the two
null vectors $v_9, v_{10}$. Finally $\ip{v_9}{v_{10}} = 2$.  
Write $r$ as a linear combination of $v_1, \dotsb, v_{10}$ in the form
\begin{equation}
r = p^{-1}(c_1 v_1 + c_2 v_2 + \dotsb + c_{10} v_{10}).
\label{eq-r}
\end{equation}
The lemma below gives enough conditions on $c_1, \dotsb, c_{10}$ to allow a computer
enumeration of all possible $(c_1, \dotsb, c_{10})$ and hence, of all possible $r$.
\begin{lemma}
One has $c_1 \dotsb, c_9 \in \cG(\leq 2)$ and $c_{10} \in \cG(\leq 9)$.
One has
\begin{equation*}
c_1 + c_2 \equiv c_3 + c_4 \equiv c_5 + c_6 \equiv c_7 + c_8 \equiv c_{10} \bmod p
\end{equation*}
and 
\begin{equation*}
 \abs{c_1}^2 + \dotsb +  \abs{c_8}^2  = 2 -  2 \op{Re}( \bar{c}_9 c_{10}) \in 
  \lbrace 2, 4, 6, 8, 10 \rbrace.
\end{equation*}
\label{l-condition-on-c}
\end{lemma}
\begin{proof}
Taking inner product of $r$ with $v_1,\dotsb, v_{10}$ we find, 
 \begin{equation*}
c_9 = \bar{p}^{-1} \ip{v_{10}}{r}, \; c_{10} = \bar{p}^{-1}\ip{v_9}{r}, \;and \;
c_j = \bar{p}^{-1}\ip{v_j}{r} \text{\; for \;} j = 1, \dotsb, 8.
 \end{equation*}
 Since $L$ is $p$-modular, $c_1, \dotsb, c_{10} \in \cG$.
Write $r$ in the coordinate system $4 D_4^{\cG} \oplus \cell$: 
 \begin{align*}
r = (c_1, c_2 + c_{10}, c_3, c_4 + c_{10}, c_5, c_6 + c_{10}, c_7, c_8 + c_{10};
r_9, c_{10} )
\end{align*}
where
\begin{align} 
r_9 = (c_9 + (p-3) c_{10} - c_2 - c_4 - c_6 - c_8)/p.
\label{eq-r9}
 \end{align}
From the definition of $D_4^{\cG}$, it follows that
$c_{2j-1} + c_{2j} + c_{10} \equiv 0 \bmod p$ for $j = 1, 2, 3, 4$.
This implies the congruences on the $c_j$'s.
\par
The bounds on norms of $c_j$'s follow from  lemma \ref{l-bounds}. As
 already noted, lemma \ref{l-bounds}(a) implies 
 $ p^{-1} \ip{r}{s_v}  \in \cG(\leq 2)$ for $v \in D$.
In particular, $c_1, \dotsb, c_8 \in \cG(\leq 2)$.
Next,  lemma \ref{l-bounds}(b) implies that
\begin{equation*}
\abs{ p^{-1} \ip{v_{9}}{r} }^2 \leq e^{2 (d_0 +  d_{v_{9}}(\tau))} \approx 9.3379,
\end{equation*}
and
\begin{equation*}
\abs{ p^{-1} \ip{v_{10}}{r} }^2 \leq e^{2 (d_0 +  d_{v_{10}}(\tau))} \approx 3.2043.
\end{equation*}
This implies that  $c_{10} = \bar{p}^{-1}\ip{v_9}{r} \in \cG(\leq 9)$
and $c_{9} = \bar{p}^{-1}\ip{v_{10}}{r} \in \cG(\leq 2)$.
 Taking norm on both sides of equation  \eqref{eq-r} and rearranging, we obtain
 \begin{equation*}
\abs{c_1}^2 + \dotsb +  \abs{c_8}^2  = 2 - 2 \op{Re}( \bar{c}_9 c_{10}).
\end{equation*}
We already know that there are a small number of possibilities for $c_9$ and $c_{10}$. Enumerating these,
we find that $\op{Re} (\bar{c}_9 c_{10} )\in  [-4, 4] \cap \Z$.
Since $\abs{c_1}^2 + \dotsb +  \abs{c_8}^2  \geq 0$, it follows that
$\op{Re} (\bar{c}_9 c_{10}) \in [ -4, 1] \cap \Z$.
The lemma follows once we argue that $\op{Re} (\bar{c}_9 c_{10} )\neq 1$.
If possible, suppose $\op{Re} (\bar{c}_9 c_{10}) = 1$. Then $\sum_{j = 1}^8 \abs{c_j}^2 = 0$,
so $c_1 = \dotsb = c_8 = 0$. The congruences satisfied by $c_j$'s now implies
$c_{10} \equiv 0 \bmod p$. Now the condition $r_9 \in \cG$ implies that $c_9 \equiv 0 \bmod p$.
But this implies that $\op{Re} (\bar{c}_9 c_{10}) \in 2\Z$ which is a contradiction.
\end{proof}
\begin{proof}[proof of theorem \ref{th-simple-mirrors}]
We use a computer program to enumerate all possible tuples $(c_1, \dotsb, c_{10})$ satisfying the conditions of 
lemma \ref{l-condition-on-c} 
and subject to the further restriction that $c_9 \in \lbrace 0, 1, p \rbrace$.
We may assume $c_9 \in \lbrace 0, 1, p \rbrace$ since it is enough to enumerate
the possible tuples $(c_1, \dotsb, c_{10})$ up to units.
Let $r = (\sum_j c_j v_j)/p$. Then $r$ is a root of $L$ if and only if 
$r_9 \in \cG$ (see equation \eqref{eq-r9}) .
We run through the possibilities for $r$ and list those for which $r_9 \in \cG$ and
$d(\tau, r^{\bot}) \leq d_0$. This produces only the unit multiples of $\lbrace s_v \colon  v\in D \rbrace$.
\end{proof}
\begin{theorem}
The $i$-reflections in the $32$ roots $\lbrace s_v \colon v \in D \rbrace$
generate $R(L)$. These generators obey the Coxeter relations dictated by
$D$ as stated in the introduction.
\label{th-32-reflections-generate-R(L)}
\end{theorem}
\begin{proof}
Write $S = \lbrace s_v \colon v \in D \rbrace$. Let $G$ denote the subgroup
of $R(L)$ generated by the reflections in $S$.
Lemma \ref{l-finite-list-of-generators-for-R(L)} provides us a finite set of roots
$S_0 \cup S_1 \cup S_2$  such that the $i$-reflections in them generate $R(L)$.
We take a root $x_0 \in S_0 \cup S_1 \cup S_2$ and try to find some $s \in S$ and $\xi \in \lbrace i, i^2, i^3 \rbrace$
 such that $x_1^{\bot} = R_{s}^{\xi} (x_0^{\bot})$ is closer to $\tau$
than $x_0^{\bot}$. We repeat this to obtain a sequence of roots $x_0 , x_1, x_2 \dotsb$.
If some $x_j$ is an unit multiple of $S$, then we say that height reduction (with respect to $\tau$)
is successful for $x_0$ and in this case, we obtain $R_{x_0} \in G$.
 \par
  In a computer calculation, height reductions is successful for most of the $123426$
  roots in $S_0 \cup S_1 \cup S_2$. For $401$ roots (all from $S_2$)
 height reductions is not successful. In these cases,
 we end up with a root  $x_j^{\bot}$ whose distance from $\tau$ cannot be
 decreased by any reflection in $S$. Let $S'$ be the set of these $401$ roots.
 To deal with these cases, by little experimentation, we found a root $y \in S_2 - S'$ 
 such that height reduction is successful for all the root in $\lbrace R_y(x) \colon x \in S' \rbrace$.
This means that $R_y \in G$ and further 
 that for each $x \in S'$, one has $R_y R_x R_y^{-1} \in G$; hence $R_x \in G$.
This proves that the reflections in $S$ generate $R(L)$. 
The Coxeter relations between these generators are consequences of the inner products
between the roots in $S$, as given in \eqref{eq-ipD}.
\end{proof}
\begin{topic}{\bf Remarks on computer calculations:} 
\label{t-computer}
In the proof of theorem \ref{th-32-reflections-generate-R(L)} we glossed over
one step. The  roots $S_0 \cup S_1 \cup S_2$
in lemma \ref{l-finite-list-of-generators-for-R(L)} are given in the coordinate
system $\bwg \oplus \cell$ while the $32$ roots in $S$ are given in the coordinate
system $4 D_4^{\cG} \oplus \cell$. So we need to find an explicit isomorphism 
from $\bwg \oplus \cell$ to $4 D_4^{\cG} \oplus \cell$. 
Our computation used the following $10$ vectors in $\bwg  \oplus \cell$:
\begin{align*}
v_1 &= \ph[  \ph p, \ph  p, \ph  \bar{p},  -\bar{p}, \ph  \bar{p},  -\bar{p},  \ph p,  \ph p, \ph \ph  2, -2]/2, \\
v_2 &=-[  \ph \bar{p}, \ph  p,  \ph p,  \ph \bar{p}, \ph  p,  -\bar{p},  -\bar{p},  \ph p, \ph \ph  2,  -2]/2, \\
v_3 &=\ph [ \ph  0, \ph  p,  \ph 0,  \ph 0, \ph  0, \ph  0,  \ph 0, \ph  p, \ph \ph  1,   -1], \\
v_4 &=-[  \ph 1, \ph  i, \ph  0,  \ph 0, \ph  1, \ph   i,  \ph 0, \ph  0, \ph \ph 1,   -1], \\
v_5 &=\ph [ \ph  \bar{p}, \ph  p, \ph  p,  -\bar{p},  \ph p,  -\bar{p},  \ph \bar{p},  \ph p, \ph \ph  2,   -2]/2, \\
v_6 &=-[ \ph  1, \ph  i, \ph  0,  \ph 0, \ph  0, \ph  0, \ph  i,  \ph 1, \ph \ph  1,   -1], \\
v_7 &=\ph [  \ph p,  \ph p,  \ph p, \ph  p, \ph  p, \ph  p, \ph  p, \ph  p, \ph \ph  2,   -2]/2, \\
v_8 &=-[  \ph \bar{p},  \ph p,  -p,  -\bar{p}, \ph  p,  -\bar{p},  -\bar{p}, \ph  p, \ph \ph  2,   -2]/2, \\
v_9 &=\ph [ \ph 3 \bar{p}, \ph 4 + p,  \ph \bar{p}, \ph  p, \ph 2+\bar{p}, \ph 1+3 i, \ph p,  \ph 4+\bar{p}, \ph \ph 4-6 i, \ph -4 \bar{p}]/2, \\
v_{10} &=-[ \ph  0, \ph  0, \ph  0,  \ph 0,  \ph 0,  \ph  0, \ph 0,  \ph 0, \ph \ph   1,  \ph  i] - v_9.
\end{align*}
One verifies that the inner products between these $10$ vectors are the same
as the $10$ vectors $(d_1, c_1, d_2, c_2, d_3, c_3, d_4, c_4, (0^8; 1, 0), (0^8; 0, 1))$
that form a basis of $4 D_4^{\cG} \oplus H$. So sending $v_1, v_2, \dotsb$
to $d_1, c_1, \dotsb$ defines an isomorphism from 
$\bwg \oplus \cell$ to $4 D_4^{\cG} \oplus \cell$. Finding the vectors $v_1, \dotsb, v_{10}$ required
considerable computation using a list of the $4320$ short vectors of $\bwg$.
%We looked for $v_1, \dotsb, v_8$ having the form $v_j = (u_j; \pm 1 , *)$ with $u_j$ is a short vector of $\bwg$
%such that they have the same inner product as $(d_1, c_1, \dotsb, d_4, c_4)$. This is a manageable computation.
%After finding $v_1, \dotsb, v_8$, one computes their orthogonal complement $H'$ in $\bwg \oplus \cell$.
%One has  $H' \simeq \cell$. Now it is not hard to find two
%primitive null vectors in $H'$ with inner product $p$. This produced $v_9, v_{10}$.
We shall omit these computational details, since,
for the purpose of proving theorem \ref{th-32-reflections-generate-R(L)},
these computations are irrelevant
 once the vectors $v_1, \dotsb, v_{10}$ have been found.
 One has to simply verify that $v_1, v_2, \dotsb$ has the same gram matrix as
 $d_1, c_1, \dotsb$. The root $y \in 4D_4^{\cG} \oplus \cell$ used in the proof 
 of theorem \ref{th-32-reflections-generate-R(L)} to perturb the $401$ elements of $S'$ is
 \begin{equation*}
 y = [1 + 2 i, 3, 1 + i, 5 + i, 1 + 2 i, 4 + i, 1 + 2 i, 4 + i, 7, - 6 i].
 \end{equation*}
 \par
All the computer calculations needed in this article were performed using the pari/gp calculator.
The calculations are contained in the file \verb|bw2.gp|, available on the website
\verb|math.iastate.edu/~tathagat/codes|. 
The calculations needed for theorem \ref{th-32-reflections-generate-R(L)}
only use exact arithmetic.
We should note that a lot of calculations performed while
trying height reduction on the $123426$ roots become redundant after the fact. 
To aid verification of our proof, 
below, we sketch the computations one needs to perform.
The $32$ roots in $S$ are named $\textsf{ss}[1], \dotsb, \textsf{ss}[32]$ in 
\verb|bw2.gp|.
The function  \verb|generate_S_all()|
generates the $123426$ roots in $S_0 \cup S_1 \cup S_2$. These are named
\verb|s_list[1], ..., s_list[123426]|.
%The function \verb|generate_S_all()| uses several other codes. This
%constitute the bulk of computations needed to verify
%theorem \ref{th-32-reflections-generate-R(L)}.
\par
The function \verb|generate_path()| 
uses  \verb|generate_S_all()| and runs the height reduction algorithm 
on the $123426$ roots. It outputs
a large file \verb|bwpath| that contains a list
of vectors   \verb| bwpath[1], ... ,bwpath[123426]|. This program takes about an hour and half to run
on a laptop. All the other codes take at most a few minutes.
Each \verb|bwpath[i]| is a string of integers $(n_1, \dotsb, n_k)$ with 
each $n_j \in \lbrace 1, \dotsb, 64 \rbrace$. 
Rename the file \verb| bwpath| as \verb| bwpath.gp| and read it into pari/gp.
For each $j$, let 
\begin{equation*}
R_j = \begin{cases}
R_{\mathsf{ss}[n_j]} & \text{\; if \;} 1 \leq n_j \leq 32, \\
R_{\mathsf{ss}[n_j - 32]}^{-1} & \text{\; if \;} 33 \leq n_j \leq 64. 
\end{cases}
\end{equation*}
The code 
\verb|verify_path_all()|
checks  that for each applying $R_k R_{k-1} \dotsb R_2 R_1$ to
\verb|s_list[i]|
produces an unit multiple of an element of $S$. 
\end{topic}
\begin{theorem}
The thirteen $i$-reflections in $s_a, s_{b_k}, s_{c_k}, s_{d_k}$ for $k = 1,2,3,4$ generate $R(L)$.
%(Here, for simplicity, we are writing $a, b_1, c_1, \dotsb$ instead of $s_a, s_{b_1}, s_{c_1}, \dotsb$).
\label{th-13-reflections-generate-R(L)}
\end{theorem}
Theorem \ref{th-13-reflections-generate-R(L)} follows quickly from 
Theorem \ref{th-32-reflections-generate-R(L)}. Before
giving the proof, we recall a definition from \cite{TB:uggr}.
\begin{definition}
Let $\lbrace x_j \colon j \in \Z/k \rbrace$ be the elements in a monoid and let $m$ be a positive integer.
Let $\cC_m \langle x_0, \dotsb, x_{k-1} \rangle$
denote the positive homogeneous relation
\begin{equation*}
x_0 x_1 \dotsb x_{m-1} = x_1 x_2 \dotsb x_m.
\end{equation*}
For example $\cC_2 \langle x, y \rangle$ (resp. $\cC_3 \langle x, y \rangle$) denote the
relations $x y =  y x$ (resp. $x y x = y x y $).  
Let $\tilde{A}_{n-1}$ be the affine Dynkin diagram of type $A_{n-1}$ with vertices labeled
by $\Z/n$. Let $\lbrace y_j \colon j \in \Z/n \rbrace$ be the  generators
for the corresponding Artin group: 
\begin{equation}
y_j y_k y_j = y_k y_j y_j \text{  if $j$ and $k$ are adjacent and \;} y_j y_k = y_k y_j
 \text{ otherwise}.
\label{eq-bc}
\end{equation}
If $y_j^2 = 1$, then we have a presentation of the Affine Weyl group of type $A_n$.
 The relation $\cC_{n-1} \langle y_1, \dotsb, y_{n} \rangle$
collapses this affine Weyl group to the spherical Weyl group or the symmetric group.
Following  \cite{CSi:26}, we call this relation {\it deflating the $n$-gon} $(y_1, \dotsb, y_{n})$.
One can verify that in the presence of the braiding and commuting relations of \eqref{eq-bc}, 
the deflation relation $\cC_{n-1} \langle  y_1, \dotsb, y_{n} \rangle$ is equivalent to 
$\cC_{n-1} \langle  y_{j+1}, \dotsb, y_{ j + n} \rangle$ for any $j$
(see \cite{TB:uggr}, lemma 4.3 (a)).
\end{definition}
\begin{proof}[proof of \ref{th-13-reflections-generate-R(L)}]
Write $r_v = R_{s_v}$.
Let $G$ be the subgroup of $R(L)$ generated by the thirteen reflections $r_a, r_{b_k}, r_{c_k}, r_{d_k}$.
One verifies that $(d_2, c_2, b_2, a, b_1, c_1, d_1, e_{1 2} )$ is an octagon in the graph $D$
and the deflation relation 
\begin{equation*}
\cC_7 \langle r_{d_2}, r_{c_2}, r_{b_2}, r_a, r_{b_1}, r_{c_1}, r_{d_1}, r_{e_{12}} \rangle
\end{equation*}
holds
in $R(L)$. This shows that $r_{e_{12}} \in G$. By $S_4$ symmetry we obtain 
$r_{e_{j k}} \in G$ for all $j, k$.
Next, one verifies that one has the octagons 
$(d_1, c_1, b_1, e_{3 4}, b_2, c_2, d_2, z)$, 
$(c_2, b_2, a, b_3, e_{2 4}, d_4, c_4, f_1)$, 
$(c_4, b_4, e_{1 3}, d_3, z, d_2, c_2, h_1)$,
$(d_2, e_{2 3}, f_3, c_4, h_3, a, b_2, g_1)$
in $D$ and deflation relation holds for each of these octagons.
Applying $S_4$ symmetry it successively follows that
$r_z, r_{f_k}, r_{h_k}, r_{g_k} \in G$ as well. Now the theorem follows from
\ref{th-32-reflections-generate-R(L)}.
\end{proof}
%
%*******************************************************************************************************
%
%
%
%
%*******************************************************************************************************
%
%
%*******************************************************************************************************
%

%
\end{document}